\documentclass[USenglish]{article}	
\usepackage[utf8]{inputenc}				%(only for the pdftex engine)
\usepackage[big,online]{dgruyter}	%values: small,big | online,print,work
\usepackage{lmodern} 
\usepackage{microtype}
\usepackage[numbers,square,sort&compress]{natbib}
\usepackage{times,amssymb,amsmath,amscd,mathrsfs}
\usepackage{amsthm,bm}
\usepackage{xcolor}
\theoremstyle{dgthm}
\newtheorem{theorem}{Theorem}
\newtheorem{corollary}{Corollary}

\theoremstyle{dgdef}

\def\u{{\bm u}}
\def\vt{{\bm v}}
\def\f{{\bm f}}
\newcommand{\Eq}[1]{(\ref{eq:#1})}

\newcommand{\vertiii}[1]{{\left\vert\kern-0.25ex\left\vert\kern-0.25ex\left\vert #1 
    \right\vert\kern-0.25ex\right\vert\kern-0.25ex\right\vert}}

\begin{document}

%%%--------------------------------------------%%%
	\articletype{Research Article}
	\received{Month	DD, YYYY}
	\revised{Month	DD, YYYY}
  \accepted{Month	DD, YYYY}
  \journalname{De~Gruyter~Journal}
  \journalyear{YYYY}
  \journalvolume{XX}
  \journalissue{X}
  \startpage{1}
  \aop
  \DOI{10.1515/sample-YYYY-XXXX}
%%%--------------------------------------------%%%

\title{Finite element formulations for Maxwell's eigenvalue problem using continuous Lagrangian
interpolations}
\runningtitle{FEM for Maxwell's EVP using continuous Lagrangian interpolations}
%\subtitle{Insert subtitle if needed}

\author[1]{Daniele Boffi}
%\ use * to mark the author as the corresponding author
\author[2]{Ramon Codina}
\author*[3]{\"{O}nder T\"{u}rk} 
\runningauthor{D.~Boffi et al.}
\affil[1]{\protect\raggedright 
King Abdullah University of Science and Technology, Thuwal 23955-6900, Saudi Arabia, and
University of Pavia, Pavia 27100, Italy}
\affil[2]{\protect\raggedright 
Universitat Polit\`ecnica de Catalunya, and
International Centre for Numerical Methods in Engineering (CIMNE), 08034 Barcelona, Spain}
\affil[3]{\protect\raggedright 
Institute of Applied Mathematics, Middle East Technical University, 06800 Ankara, Turkey, e-mail:onder.turk@yandex.com }	
%\communicated{...}
%\dedication{...}
	
\abstract{We consider nodal-based Lagrangian interpolations for the finite element approximation of the Maxwell eigenvalue problem. The first approach introduced is a standard Galerkin method on Powell-Sabin meshes, which has recently been shown to yield convergent approximations in two dimensions, whereas the other two are stabilized formulations that can be motivated by a variational multiscale approach. For the latter, a mixed formulation equivalent to the original problem is used, in which the operator has a saddle point structure. The Lagrange multiplier introduced to enforce the divergence constraint vanishes in an appropriate functional setting. The first stabilized method we consider consists of an augmented formulation with the introduction of a mesh dependent term that can be regarded as the Laplacian of the multiplier of the divergence constraint. The second formulation is based on orthogonal projections, which can be recast as a residual based stabilization technique. We rely on the classical spectral theory to analyze the approximating methods for the eigenproblem. The stability and convergence aspects are inherited from the associated source problems. We investigate the numerical performance of the proposed formulations and provide some convergence results validating the theoretical ones for several benchmark tests, including ones with smooth and singular solutions.
}


\maketitle
	
\section{Introduction}

The main object of this paper is to approximate the Maxwell eigenvalue problem (EVP) which, for instance, can be considered as the problem of determining resonances in a perfectly conducting cavity and the associated nontrivial time harmonic electric field (see, e.g., \cite{caorsi2000,perugias1,barrenechea2014}). Defined on a bounded polyhedral domain $\Omega$ in $\mathbb{R}^d$, $d=2,3$, the EVP we consider consists of finding $[\u, \lambda]$, where $\lambda \in \mathbb{R}$, such that 
\begin{align}
\label{eq:max_evp1}
\begin{cases}
  \mu  \nabla  \times   \nabla \times \u    = \lambda \u \quad & \hbox{in}~\Omega, \\
  \nabla \cdot \u =0 \quad & \hbox{in}~\Omega, \\
	{\bm n}  \times \u= \bm {0} \quad & \hbox{on}~\partial\Omega, 
\end{cases}
\end{align}
where $\mu$ is a positive parameter taken in accordance with the physical assumptions on the problem setting. 

This EVP is of fundamental importance in computational electromagnetism. An active intense research is ongoing in the development of finite element (FE) methods that are capable of correctly approximating the constitutive and topological relations by this and other Maxwell systems. It is well known that the use of the curl conforming N\'ed\'elec or edge elements (rotated Raviart-Thomas elements in two dimensions) providing continuity of the tangential vector components while leaving the normal field components discontinuous across interfaces, results in convergent (spurious-free) approximations.  The N\'ed\'elec elements provide a natural basis for the finite element methods that satisfy the discrete inf-sup condition \cite{monk2001, boffi2010} and, furthermore, Dirichlet conditions on the tangent component of the vector unknown are easy to impose, which is not as clear using nodal elements; this issue is partially touched in this paper. However, as the continuity of the tangential field is inherent along the boundaries, a normal continuity may also be concerned and, moreover, there are situations where the edge elements do not provide optimal implementation and approximation properties. Thus, there are evident reasons to require the use of Lagrange finite elements with low order interpolations and less constraints on the problem domain discretization. In search for this, an equivalent problem to \Eq{max_evp1} can be obtained by reformulating it as a saddle point problem by the enforcement of the divergence constraint using a Lagrange multiplier $p$. In this case, the given system is governed by the Euler-Lagrange equations given as:  find $[\u, p, \lambda]$, where $\lambda \in \mathbb{R}$, such that 
\begin{align}
\label{eq:max_evp2}
\begin{cases}
  \mu  \nabla  \times   \nabla \times \u  + \nabla p = \lambda \u \quad & \hbox{in}~\Omega, \\
  \nabla \cdot \u =0 \quad & \hbox{in}~\Omega, \\
	{\bm n}  \times \u= \bm {0} ~\hbox{and}~ p=0\quad & \hbox{on}~\partial\Omega. 
\end{cases}
\end{align}
The mixed variational form of this problem is also known as the Kikuchi formulation  \cite{boffi2010,kikuchi1987}. 

On the other hand, it is well known that the use of standard nodal continuous Galerkin finite element schemes  to approximate the standard Maxwell systems has the pathology of producing spurious or non physical solutions even on smooth domains. There are many strategies consisting  of penalization or  regularizing the operator by adding a term containing the divergence. We refer for instance to \cite{badia-codina-2010-1,badia-codina-2009-2,boffi2010,boffi-guzman-neilan-2022,christiansen2018}, and the inclusive list of references therein on the subject. Also, mixed methods consisting of inf-sup stable elements using nodal continuous Lagrange elements of any order for the vector field together with a piecewise constant approximation for the multiplier have been proposed in recent works \cite{duan2019, du2020}.

In this paper, our main interest is to approximate the eigenvalues and eigenfunctions of the Maxwell operator without spurious solutions, using continuous Lagrange finite elements. Firstly we consider the standard Galerkin approximation on Powell-Sabin triangulations, where the convergence of the eigenvalues is provided in a recent work \cite{boffi-guzman-neilan-2022} (see also the three dimensional generalization in~\cite{boffi-gong-guzman-neilan-2023}). Next, the mixed finite element formulations based on two stabilized forms are presented; a so called augmented formulation, and a formulation that is based on projections. The first stabilized method provides pressure stability by inclusion of a least squares form of the  divergence constraint introduced for the corresponding source problem in \cite{badia-codina-2009-2}. The second approach is based on stabilizing the divergence and gradient components that are orthogonal to the associated finite element spaces, which is analyzed in \cite{badia-codina-2011-1}. A reinterpretation of these two methods in a unified framework with a brief analysis of their key properties has also been given in \cite{badia-codina-2010-1}.  

The outline of the paper is as follows. In Section~2 we briefly describe the standard Galerkin method on Powell-Sabin meshes, whereas the two stabilized formulations are described in Section~3. While the EVP is directly analyzed for the Galerkin method, convergence results for the stabilized formulations rely on the approximation of the source problem and the classical spectral theory, which is applied in Section~4. The main objective of this work is to check and compare the performance of nodal based formulations for the problem at hand, and this is done in Section~5. Finally, some conclusions are drawn in Section~6. 

\section{The standard Galerkin approximation with Powell-Sabin meshes}
\label{sec:galerkinps}

The variational formulation of \Eq{max_evp1} is given as follows: find $\u \in H_0({\textbf{curl}},\Omega)$, $\u \neq {\bm 0}$, and $\lambda \in \mathbb{R}$ satisfying
\begin{align}
\label{eq:stvarpr}
\mu  (\nabla \times \u, \nabla \times \vt)  = \lambda (\u,\vt) , \quad \forall   \vt  \in  H_0({\textbf{curl}},\Omega),
\end{align}
 where   $H({\textbf{curl}},\Omega)= \{ \vt \in L^2(\Omega)^d : \nabla \times \vt \in   L^2(\Omega)^d \}$, $H_0({\textbf{curl}},\Omega)= \{ \vt \in H({\textbf{curl}},\Omega) : {\bm n}  \times \u= \bm {0} ~\text{on} ~\partial \Omega\}$, and $(\cdot, \cdot)$ denotes the $L^2$-inner product defined over $\Omega$. Note that there is no need to enforce the divergence free condition for $\u$, as taking the divergence on both sides of the first equation in \Eq{max_evp1} directly yields that this field is divergence free.

For the FE approximation to this problem, let $\mathcal{T}_h$ be a partition of the problem domain $\Omega$ into a set of element domains $\{K\}$. As usual, $h$ denotes the characteristic mesh size of the partition, here taken as $h= \max_{K \in \mathcal{T}_h} h_K$, where $h_K$ is the diameter of the element $K$. Since our main interest is in the nodal approximations, we define the space of piecewise continuous  polynomials on $\Omega$ as 
$$\mathcal{N}_k(\Omega)  =\{ v_h \in \mathcal{C}^0 (\bar{\Omega}) : v_h \vert_K \in  \mathcal{P}_k(K), \quad \forall K \in  \mathcal{T}_h\},$$ 
where $\mathcal{P}_k(K)$ denotes the space of polynomials of degree at most $k$ defined on $K$. For the components of the vector fields as well as the scalar fields we will make use of these $H^1(\Omega)$-conforming approximating spaces in which every function can be determined uniquely by its values on the set of nodes of the defining elements. For all the analysis given in this work we assume for the sake of simplicity that the partitions are quasi-uniform.  

The Galerkin discretization on a finite dimensional space ${\cal V}_h$ of partition size $h$, can be written as: find nonzero $\u_h \in {\cal V}_h$,   and $\lambda_h \in \mathbb{R}$ such that 
\begin{align}
\label{eq:galerkin}
\mu  (\nabla \times \u_h, \nabla \times \vt_h)  = \lambda (\u_h,\vt_h) , \quad \forall   \vt_h  \in  {\cal V}_h .
\end{align}
As we have already mentioned, it is known that when ${\cal V}_h$ is taken as the $H^1(\Omega)$-conforming Lagrange finite element space spurious solutions may exist for an arbitrary mesh size (see \cite{boffi2010,boffi-guzman-neilan-2022}, and the references therein). An example to such pathology will be given in Section \ref{sec:numres}. On the other hand, it has been proved in \cite{boffi-guzman-neilan-2022} that the standard formulation when implemented by the use of linear Lagrange finite elements on Powell-Sabin triangulations yields convergence of the eigenvalues to the true ones; the analysis presented in this reference is not based on the approximation properties of the method for the source problem, since it is not well posed. We will also include a numerical evidence for this in Section \ref{sec:numres}.

\section{Stabilized formulations}
\label{sec:stab}

The two stabilized formulations we shall consider are based on the Kikuchi formulation \Eq{max_evp2} of the EVP.  There are essentially two alternatives for the variational form of this formulation, depending on whether the term with $p$ is integrated by parts or not, as this inherently implies two possible choices for the functional framework of the problem~\cite{badia-codina-2009-2, badia-codina-2010-1}. Consistently with the Galerkin approximation described earlier, we consider only the so called curl formulation, in which the pressure gradient is not integrated by parts and the space where the solution is sought is ${\cal X} =  H_0({\textbf{curl}},\Omega)   \times H^1_0(\Omega)$. The problem then reads: find  $[\u,p]  \in {\cal X}$ and $\lambda \in \mathbb{R}$ such that 
\begin{align}
\label{eq:varevp}
B_{ }([\u,p],[\vt,q]) = \lambda (\u,\vt) , \quad \forall  [\vt,q] \in  {\cal X},
\end{align}
where 
\begin{align*}
B_{ }([\u,p],[\vt,q]) = \mu  (\nabla \times \u, \nabla \times \vt) + (\nabla p, \vt )  - (\nabla q, \u ) .  
\end{align*}

Let ${\cal V}_h\subset {\cal V} := H_0({\textbf{curl}},\Omega) $ and ${\cal Q}_h \subset {\cal Q} := H^1_0(\Omega)$ be the FE spaces to approximate $\u$ and $p$, respectively. The Galerkin approximation of the variational problem in ${\cal X}_h = {\cal V}_h\times {\cal Q}_h \subset {\cal X}$ is given as: find  $[\u_h,p_h]  \in {\cal X}_h$ and $\lambda_h \in \mathbb{R}$ such that   
\begin{align}
\label{eq:varevpG}
B_{ }([\u_h,p_h],[\vt_h,q_h]) = \lambda_h (\u_h,\vt_h) , \quad \forall  [\vt_h,q_h] \in  {\cal X}_h.
\end{align}
The corresponding source problem is well posed in spaces ${\cal V}_{h}$ and ${\cal Q}_{h}$ if they satisfy the discrete inf-sup condition
\begin{align}
\inf_{p_h\in {\cal Q}_{h}}\sup_{\vt_h\in {\cal V}_{h}} \frac{(\nabla p_h,\vt_h)} {\Vert p_h \Vert_{\cal Q} \Vert \vt_h \Vert_{\cal V}} \geq K_b,\label{eq:littleinfsup-h}
\end{align}
($\Vert\cdot\Vert_{\cal B}$ standing for the norm in a space $\cal B$)
for a certain constant $K_b > 0$, and this is a sufficient condition for the EVP to be well-posed. Examples of pairs of spaces satisfying this condition are those based on N\'ed\'elec's elements to construct ${\cal V}_{h}$ and nodal Lagrangian continuous elements to construct ${\cal Q}_{h}$. However, as it has been said, we are interested in nodal-based interpolations, case in which ${\cal V}_h$ and ${\cal Q}_h$ fail to fulfill condition~\Eq{littleinfsup-h}. As explained in the previous section, there is also the possibility to use the Galerkin method with nodal elements without introducing the Lagrange multiplier $p$ if the FE mesh is of Powell-Sabin type.

The alternative is to switch to a stabilized FE approximation to approximate the eigenproblem given in \Eq{varevp}. We consider in this paper two options, and both can be written in a unified manner as follows: find  $[\u_h,p_h]  \in {\cal X}_h$ and $\lambda_h \in \mathbb{R}$ such that 
\begin{align}
\label{eq:stabform}
B_{\text{S}}([\u_h,p_h],[\vt_h,q_h]) = \lambda_h (\u_h,\vt_h) , \quad \forall  [\vt_h,q_h] \in  {\cal X}_h,
\end{align}
where $B_{\text{S}}([\u_h,p_h],[\vt_h,q_h])$ is defined as
\begin{align*}
B_{\text{S}}([\u_h,p_h],[\vt_h,q_h]) &= B([\u_h,p_h],[\vt_h,q_h])  \\
 &+  \sum_{K} \tau_{p} (\tilde{P}(\nabla p_h), \tilde{P}(\nabla q_h))_K \\ &+ \sum_{K} \tau_{\u}(\tilde{P}(\nabla \cdot \u_h), \tilde{P}(\nabla \cdot \vt_h))_K. 
\end{align*}
In the above equation  $\sum_{K}$ signifies the summation over all elements $K$ of the partition, and $(\cdot,\cdot)_K$ denotes the $L^2(K)$-inner product. The stabilization parameters are defined as follows 
$$\tau_{p}=c_p\frac{\ell^2}{\mu}, \quad \tau_{\u}= c_{\u} \mu \frac{h^2}{\ell^2},$$
where  $c_p$ and $c_{\u}$ are appropriately chosen algorithmic constants and $\ell$ is a length scale of the problem. It is observed that these parameters are constant in the case of quasi-uniform FE partitions (i.e., using the same $h$ for all elements) and constant $\mu$. Otherwise, they should be computed element-wise; this is why we have introduced the summation over the elements, which is in fact not needed in our case. 

The method depends on the projection $\tilde{P}$ which can be either the identity $I$ or the orthogonal projection to the finite element space, $P_h^\bot$, leading to the augmented (AG) and the orthogonal subgrid scale (OSGS) formulations, respectively. This orthogonal projection can be computed as $P_h^\bot = I - P_h$, $P_h$ being the $L^2(\Omega)$-projection onto the FE space, either ${\cal V}_h$ or ${\cal Q}_h$; we have not distinguished these two possibilities, being clear by the context the projection to consider. The bilinear form $B_{\text{S}}$ in \Eq{stabform} will be denoted by $B_{\text{AG}}$ when $\tilde{P}=I$, and by $B_{\text{OSGS}}$ when $\tilde{P}=P_h^\bot$.

In the case of the AG formulation, the first term introduced (that involving $\nabla p_h$) consists of adding a penalization term of the form $-\frac{\ell}{\mu} \Delta p$ in the divergence-free condition for $\u$~\cite{badia-codina-2009-2}. Obviously, this penalty is exact for the continuous problem, since the homogeneous boundary condition for $p$ implies that $p=0$ almost everywhere in $\Omega$. Indeed, in the problem
\begin{align*}
  \mu  \nabla  \times   \nabla \times \u +\nabla p   = \lambda \u \quad & \hbox{in}~\Omega, \\
 -\frac{\ell}{\mu} \Delta p+  \nabla \cdot \u =0 \quad & \hbox{in}~\Omega, \\
	{\bm n}  \times \u= \bm {0} \quad & \hbox{on}~\partial\Omega, 
\end{align*}
if we take the divergence of the first equation and use the second we get that $\Delta p = \lambda \nabla\cdot \u = \lambda \frac{\ell}{\mu} \Delta p$, so that $\Delta p = 0$ as soon as we choose $\ell$ such that $\lambda \frac{\ell}{\mu} \not = 1$, i.e., $\frac{\mu} {\ell}$ is not an eigenvalue of Maxwell's EVP.

On the other hand, the OSGS formulation is a residual-based stabilized discretization where only the components of the divergence and gradient terms in \Eq{stabform} that are orthogonal to the corresponding finite element spaces are stabilized \cite{badia-codina-2011-1, badia-codina-2010-1}. The orthogonal projections onto the corresponding space can be computed iteratively or treated implicitly. We will follow the latter option in the numerical results presented for this study. 

In order to prove that the solutions of the stabilized formulations converge to the solutions of the continuous problem, we apply the classical spectral approximation theory relying on the convergence of the associated source problem for each formulation. For this reason, let us consider the continuous source problem associated to the Maxwell EVP we deal with in this work, which reads as follows: given a solenoidal vector $\f \in  L^2(\Omega)^d$, find  $[\u,p]  \in {\cal X}$ such that 
\begin{align}
\label{eq:source}
B_{ }([\u,p],[\vt,q]) =  (\f,\vt) , \quad \forall  [\vt,q] \in  {\cal X}.
\end{align}
 
As for the EVP, the stabilized formulations for the source problem can be jointly written as:  find  $[\u_h,p_h]  \in {\cal X}_h$ such that 
\begin{align}
\label{eq:stabsource}
B_{\text{S}}([\u_h,p_h],[\vt_h,q_h]) =  (\f,\vt_h) , \quad \forall  [\vt_h,q_h] \in  {\cal X}_h,
\end{align}
where $B_{\text{S}}$ denotes either $B_{\text{AG}}$, the AG form,  or $B_{\text{OSGS}}$, the OSGS form, as defined previously. 

The stability and convergence of the AG and OSGS formulations for the source problems are analyzed in \cite{badia-codina-2009-2} and \cite{badia-codina-2011-1}, respectively. We briefly list the results we need for the EVP in the sequel and refer to the associated papers for detailed discussions. Despite both the AG and the OSGS yield good results, the norm in which stability and convergence can be proved is weaker for the latter than for the former.

Let us start with the AG formulation. The norm in which the numerical analysis of the source problem can be done is 
\begin{align}
 \Vert [\vt,q] \Vert^2_{\rm AG} :=  \mu \Vert\nabla\times \vt \Vert^2_{L^2(\Omega)} + \frac{\mu}{\ell^2} \Vert \vt \Vert^2_{L^2(\Omega)}
+ \frac{\ell^2}{\mu} \Vert \nabla p \Vert^2_{L^2(\Omega)}, \label{eq:agnorm}
\end{align}
which is a norm in ${\cal X}$ with adequate scaling coefficients. 

In the sequel, $\lesssim$ denotes an inequality up to a positive constant that is independent of the mesh size and the problem coefficients. 

The following results are proved in \cite{badia-codina-2009-2}:

\begin{theorem}\label{th:th1} Suppose that both ${\cal V}_{h}$ and ${\cal Q}_{h}$ are constructed using continuous nodal based interpolations, each of arbitrary degree. Then, problem \Eq{stabsource} (with $B_{\rm S} = B_{\rm AG}$) is well posed, in the sense that it admits a unique solution $[\u_h,p_h]\in {\cal V}_{h}\times {\cal Q}_{h}$ that satisfies
\begin{align*}
\Vert [\u_h,p_h] \Vert_{\rm AG} \lesssim \Vert {\bm f}\Vert_{V'}.
\end{align*}
Furthermore, $[\u_h,p_h]$ converges optimally as $h \to 0$ to the solution $[\u,p]\in V \times Q$ of the continuous problem \Eq{source}, in the following sense:
\begin{align}
& \Vert [\u - \u_h, p - p_h] \Vert_{\rm AG} \nonumber \\
& \qquad  \lesssim  \inf_{[ \vt_h, q_h] \in {\cal V}_{h}\times {\cal Q}_{h}} \Vert [\u - \vt_h, p - q_h] \Vert_{\rm AG} 
+ \frac{\nu^{1/2}}{\ell}\left( \sum_K h_K \Vert \u - \vt_h\Vert^2_{L^2(\partial K)}\right)^{1/2}.\label{eq:errest2}
\end{align}
\end{theorem}

The error estimate \Eq{errest2} is optimal for smooth solutions, that is when $\u$ belongs to $H^r(\Omega)^d$ for $r\ge1$.
In this case we have
\begin{equation}
\label{eq:estimatesmooth}
\Vert [\u-\u_h, p-p_h] \Vert_{\text{AG}} \lesssim  h^{t-1}  \Vert\u \Vert_{H^t(\Omega)},
\end{equation}
where $t= \min\{r,k+1\}$, $k$ being the order of the FE interpolation.

In the case of solutions with Sobolev regularity $0 < r < 1$, it is shown in~\cite{badia-codina-2009-2} that the convergence is also optimal if the FE meshes are able to interpolate optimally scalar functions of Sobolev regularity $r+1$, whose gradients are components of $\u$. This happens for example if the FE meshes are of Powell-Sabin type (see \cite{badia-codina-2009-2} and references therein for further discussion). More precisely, if ${\cal V}_h$ consists of functions whose gradients are in $\mathcal{N}_k(\Omega)^d  \cap  H_0({\textbf{curl}},\Omega)$, then (see \cite{badia-codina-2009-2}, Corollary 3.12): 
% report the splitting of regular and singular part
$$\Vert[ \u-\u_h, p-p_h] \Vert_{\text{AG}} = O(h^{t-\epsilon}),$$ 
for any $\epsilon \in \left] 0, t-1/2\right[$ for $t= \min\{r,k\}$. This condition is not used in \cite{Bonito2011} (in the context of the EVP), i.e., there is no need to use Powell-Sabin meshes in the formulation proposed there; however, this is at the expense of loosing optimality in the FE order of convergence. In fact, the formulation proposed in this reference with $\alpha = 1$ (a possibility not allowed in \cite{Bonito2011}) corresponds to the AG formulation we are considering. 

For the EVP analysis we will need to relate the $L^2(\Omega)$-norm of the error $\u-\u_h$ to the $L^2(\Omega)$-norm of the forcing term $\bm f$. This follows from Theorem~\ref{th:th1} assuming enough regularity in the domain $\Omega$ (see, e.g., \cite{CoDa,bonito-et-al-2013}).

\begin{corollary}\label{th:cor1}
Assume that the domain $\Omega$ is such that the following elliptic regularity property holds:
\begin{align}
\mu^{1/2} \Vert \nabla \times \u \Vert_{H^s(\Omega)} 
+\frac{\mu^{1/2}}{\ell} \Vert \u \Vert_{H^s(\Omega)} 
\lesssim \frac{\ell}{\mu^{1/2}} \frac{1}{\ell^{s}} \Vert \f\Vert_{L^2(\Omega)},\label{eq:ellreg}
\end{align}
for some $s > 0$. Then, there holds
\begin{align}
\frac{\mu^{1/2}}{\ell} \Vert \u - \u_h \Vert_{L^2(\Omega)}  \lesssim \frac{\ell}{\mu^{1/2}}  \left( \frac{h}{\ell}\right)^s  \Vert \f \Vert_{L^2(\Omega)}. \label{eq:l2l2}
\end{align}
\end{corollary}
\begin{proof}
The proof follows directly form standard interpolation estimates and the elliptic regularity property \Eq{ellreg}:
\begin{align*}
\frac{\mu^{1/2}}{\ell} \Vert \u - \u_h \Vert_{L^2(\Omega)} & \lesssim \Vert [\u - \u_h, 0 ] \Vert_{\rm AG} \\
&\lesssim \inf_{\vt_h \in {\cal V}_h} \Vert [\u - \vt_h, 0 ] \Vert_{\rm AG} +  \inf_{\vt_h \in {\cal V}_h} \frac{\mu^{1/2}}{\ell} \sum_K h^{1/2} \Vert \u - \vt_h\Vert_{L^2(\partial K)}  \\
& \lesssim 
\mu^{1/2}  h^s \Vert \nabla \times \u \Vert_{H^s(\Omega)} 
+\frac{\mu^{1/2}}{\ell} h^s \Vert \u \Vert_{H^s(\Omega)},
\end{align*}
from where \Eq{l2l2} follows using \Eq{ellreg}.
\end{proof}

Next, for the OSGS formulation we define the mesh dependent norm:
\begin{align}
 \Vert [\vt,q] \Vert^2_{\rm OSGS} :=  \mu \Vert\nabla\times \vt \Vert^2_{L^2(\Omega)} + \frac{\mu}{\ell^2} \Vert \vt \Vert^2_{L^2(\Omega)}
+ \frac{\ell^2}{\mu} \Vert P_h^\bot(\nabla p) \Vert^2_{L^2(\Omega)}
+ \frac{h^2}{\mu} \Vert P_h(\nabla p) \Vert^2_{L^2(\Omega)}. \label{eq:osgsnorm}
\end{align}
This norm is weaker than \Eq{agnorm} because the component of the pressure gradient in ${\cal V}_h$ is multiplied by $h^2$ instead of $\ell^2$. Nevertheless, the numerical results obtained in \cite{badia-codina-2011-1} showed that the OSGS formulation is as stable as the AG one; we shall corroborate this fact in this paper in the context of Maxwell's EVP. Stability and convergence is proved in \cite{badia-codina-2011-1}; because of the proof-technique employed, the statements of these results slightly differ from those in Theorem~\ref{th:th1}, but the essence is the same, namely, stability and optimal convergence:

\begin{theorem}\label{th:th2} Suppose that both ${\cal V}_{h}$ and ${\cal Q}_{h}$ are constructed using continuous nodal based interpolations of arbitrary degree each. Then, problem \Eq{stabsource} (with $B_{\rm S} = B_{\rm OSGS}$) is well posed, in the sense that 
\begin{align*}
\inf_{[\u_h,p_h]\in {\cal V}_h\times {\cal Q}_h} \sup_{[\vt_h,q_h]\in {\cal V}_h\times {\cal Q}_h} \frac{B_{\rm OSGS} ([\u_h,p_h],[\vt_h,q_h])}{\Vert [\u_h,p_h] \Vert_{\rm OSGS}\Vert [\vt_h,q_h] \Vert_{\rm OSGS}} \geq K_{B_{\rm OSGS}} > 0.
\end{align*}
Furthermore, $[\u_h,p_h]$ converges optimally as $h \to 0$ to the solution $[\u,p]\in V \times Q$ of the continuous problem \Eq{source}, in the following sense:
\begin{align}
& \Vert [\u - \u_h, p - p_h] \Vert_{\rm OSGS} 
  \lesssim  \inf_{q_h\in  {\cal Q}_{h}} \Vert [\u - P_{{\cal V}_h}(\u), p - q_h] \Vert_{\rm OSGS} 
+ \inf_{\vt_h\in  {\cal V}_{h}} \mu^{1/2} \Vert \nabla \times\u - \vt_h\Vert,\label{eq:errest3}
\end{align}
where $P_{{\cal V}_h}$ is the $L^2(\Omega)$-projection onto ${\cal V}_h$.
\end{theorem}

This result corresponds to Theorems 3.3 and 3.4 in \cite{badia-codina-2011-1}. Note that in this reference the possibility of using discontinuous interpolations and variable physical properties is taken into account, whereas here we are considering continuous interpolations and a constant $\mu$. Note also that the convergence result obtained is optimal. The same comments as for the AG formulation regarding the regularity of the continuous solution apply in this case.

As for the AG formulation, we also have the following corollary.

\begin{corollary}\label{th:cor2}
Under the assumptions of Corollary~\ref{th:cor1}, estimate \Eq{l2l2} also holds for the OSGS formulation.
\end{corollary}

\begin{proof}
Again this result is an immediate consequence of standard interpolation estimates and the elliptic regularity property \Eq{ellreg}, now using the $L^2(\Omega)$ and $H^1(\Omega)$ stability of the $L^2(\Omega)$-projection $P_{{\cal V}_h}$.
\end{proof}

\section{Numerical analysis of the eigenvalue problem}

For the standard Galerkin method presented in Section~2, the analysis of the EVP using 2D Powell-Sabin meshes is directly presented in~\cite{boffi-guzman-neilan-2022}, without relying on the approximation properties of the formulation for the source problem, since this would be singular. However, for the stabilized formulations described in Section~3 we can apply the general strategy of proving convergence of eigenvalues and eigenfunctions using the results obtained for the source problem. This is what we do next. We assume in what follows that the polynomial order of the FE interpolation is higher or equal than the regularity of the solution. 

As usual in the FE analysis of spectral problems \cite{boffi2010} (see also \cite{turkbofficodina2016}), having the existence and uniqueness of solutions to \Eq{source} and \Eq{stabsource}, we can define the solution operators  $T, T_h:L^2(\Omega)^d\rightarrow L^2(\Omega)^d$  so that, for any $\f \in L^2(\Omega)^d$,  $T \f = \u$ and $T_h \f = \u_h$  are the vector field components of the solutions to \Eq{source} and \Eq{stabsource}, respectively. From the convergence results of the source problems presented in the previous section, we can establish the following operator convergence
\[
\Vert T-T_h \Vert_{{\cal L}(L^2(\Omega)^d)}\to 0 \quad \hbox{as}~h\to 0 .
\]
This sets forth that the solutions of the discrete
problem \Eq{stabform} converge to those of \Eq{max_evp2} with no
spurious solutions. In particular, the convergence analysis follows along the lines of the abstract Babu\v ska--Osborn theory \cite{BaOs,boffi2010}.

We recall the main results related to the approximation of the eigensolutions in the following theorems. The first theorem states the convergence of the eigensolutions and the absence of spurious modes.

\begin{theorem}
Under the regularity assumptions of Corollary~\ref{th:cor1}, let $\lambda$ be an eigenvalue of~\eqref{eq:varevp} with multiplicity $m$. Then there are exactly $m$ eigenvalues of~\eqref{eq:varevpG}, counted with their multiplicities, that converge to $\lambda$. Moreover, given a generic compact set $K$ in the real line that does not contain any eigenvalues of~\eqref{eq:varevp}, for $h$ small enough there are no discrete eigenvalues of~\eqref{eq:varevpG} that are in $K$.
\label{th:1}
\end{theorem}

\begin{proof}
From the results of Corollaries~\ref{th:cor1} and~\ref{th:cor2}, we obtain that for all $\f\in L^2(\Omega)^d$
\[
\|\u-\u_h\|_{L^2(\Omega)}=\|T\f-T_h\f\|_{L^2(\Omega)}\lesssim  \varphi(h) \|\f\|_{L^2(\Omega)},
\]
with $\varphi(h)\to 0$ as $h\to 0$, that is, we have the convergence in norm $\Vert T-T_h \Vert_{{\cal L}(L^2(\Omega)^d)}\to 0 $ as $h$ tends to zero. From the standard Babu\v ska--Osborn theory~\cite{BaOs}, this implies the theorem.
\end{proof}

The second theorem states the rate of convergence of eigenvalues and eigenfunctions.

\begin{theorem}

Let $\lambda$ be an eigenvalue of~\eqref{eq:varevp} with multiplicity $m$ and with an eigenspace composed of eigenfunctions $\u$ with the following regularity for some $r>0$
\[
\aligned
&\u\in H^r(\Omega),\\
&\nabla\times\u\in H^r(\Omega).
\endaligned
\]
Let us denote by $\lambda_h^i$, $i=1,\dots,m$, the $m$ discrete eigenvalues corresponding to $\lambda$ according to Theorem~\ref{th:1}. Then we have the following error estimates:
\[
\aligned
&|\lambda-\lambda_h^i|=O(h^{2r}),&&i=1,\dots,m,\\
&\hat\delta(E,E_h)=O(h^r),
\endaligned
\]
where $E$ is the eigenspace associated with $\lambda$, $E_h$ is the direct sum of the eigenspaces associated with $\lambda^i_h$, $i=1,\dots,m$, and $\hat\delta$ denotes the gap between Hilbert subspaces in the energy norm.
\end{theorem}

\begin{proof}
Again, this result follows from the standard Babu\v ska--Osborn theory~\cite{BaOs} and the error estimates available for the source problem.
\end{proof}

The result about the eigenfunctions convergence can be made more explicit by interpreting the definition of gap as follows. Let $\u$ be an eigenfunction associated with $\lambda$ and denote by ${\bm\phi}^1_h, \ldots ,{\bm \phi}^m_h$ the eigenfunctions associated with the $m$ discrete eigenvalues converging to $\lambda$. Then there exists a linear combination $\u_h\in\hbox{\rm span}\{{\bm\phi}^1_h,\ldots,{\bm\phi}^m_h\}$ such that
\[
\vertiii{\u-\u_h}=O(h^r),
\]
where $\vertiii{\cdot}$ means $\|\cdot\|_{\text{AG}}$ or $\|\cdot\|_{\text{OSGS}}$ for the AG or OSGS formulation, respectively.

\section{Numerical results}
\label{sec:numres}

We present three numerical tests for the approximation of the Maxwell EVP with $\mu=1$ on three different domains, namely, a square domain, a flipped L-shape domain, and a cracked square domain, all of them in 2D. We consider the square domain for validation purposes, since the analytical solutions are available. The L-shaped domain consists of non-smooth solutions and hence it is a standard benchmark to test the methodology for singular problems. The cracked square domain contains a slit which makes it a distinguished candidate for serving as a challenging task with a solution that exhibits a strong singularity. 

The standard Galerkin method \Eq{galerkin}, from now on SG, and the stabilized (AG and OSGS) formulations \Eq{stabform} are implemented using criss-cross (CC) and Powell-Sabin (PS) meshes. We have also considered a sequence of uniform right diagonal meshes for a distinct case (see Section \ref{subsec:sqd}). The CC mesh sequences are used as it is well known that the gradients are well represented by the given interpolations, although there is no theoretical support for their performance in the Maxwell problem. In fact, in~\cite{badia-codina-2009-2} it was shown through numerical experiments that they provide good results for the source problem using the AG formulation, and here we will test these meshes for the EVP. The PS meshes are chosen due to their convergence results as we have already stated. For all the unknowns we have employed equal order (for both $\u_h$ and $p_h$) of linear ($P_1$) and quadratic ($P_2$) interpolations, even though much of the focus is placed on the results obtained from the former. 

All the results we present below have been obtained by means of computer programs written by us using Matlab. The eigenvalues are obtained by its built-in function \emph{eigs} that calculates a subset of eigensolutions of a (generalized) matrix eigenvalue system. For the SG formulation, the discrete spectrum consists of a number of zeros; the discrete frequencies approximating zero are eliminated to present the first nonzero values. The correctness of these approximations is verified using Matlab's \emph{eig} function calculating all of the discrete eigenvalues for sizes that fit well in memory. However, the restriction of eliminating a huge number of (machine) zeros from the approximate spectrum is alleviated when considering the stabilized formulations which, by construction, result in strictly positive eigenvalues for all the cases considered. 

The values of the algebraic constants in the definition of the stabilization parameters, $c_{\u}$ and $c_p$, and the characteristic length $\ell$ can be taken in a wide range, influencing the accuracy while preserving similar convergence behaviors. The specific values taken for the simulations are given for each test domain individually. We denote by $N$ the  number of divisions in each direction for the square domain, and the number of division in one of the short edges for the L-shaped domain. The tables we provide in the sequel list the approximated eigenvalues together with their rates of convergence towards the reference values indicated in parentheses. 

Regarding the boundary conditions, we have considered ${\bm n}\times \u_h = {\bf 0}$ on the boundaries. At convex corners of the computational domain we have prescribed both components of $\u_h$ to zero. However, the situation is more delicate at re-entrant corners; the way to impose boundary conditions there is explained for the two examples in which this situation is found.

\subsection{The square domain}
\label{subsec:sqd}

For the first numerical test, we consider an approximation of Maxwell's EVP on the square  $\Omega_1=\left] 0, \pi\right[^2$ (see Figure~\ref{fig:sq_meshp1p2}). In this case, the exact solutions are known, and the eigenvalues are given as $\lambda_{m,n}=m^2+n^2$, where $m,n=0,1,\ldots$, and $m+n \neq 0$. It is well known that SG formulations may result in unphysical values even for problems with smooth solutions unless certain conditions on the mesh topology are satisfied. More specifically, adequate gradients of the solution should be provided by the discrete space to ensure that the zero frequency is exactly approximated by vanishing discrete eigenvalues. An example to the existence of spurious eigenvalues has already been presented in \cite{boffi2010}, in which the SG scheme for the discretization of the problem on $\Omega_1$ has been used with a sequence of CC meshes. Besides realization of this issue, in order to compare the SG formulation with the stabilized ones on the same meshes, we compute the corresponding eigenvalues using CC meshes and PS meshes, and list the results in the following. Table \ref{tab:o1_p1_standard_CC} lists the first 17 approximate eigenvalues using the SG formulation and $P_1$ elements on CC meshes.  As it is evident from the negative rates, some listed limit values are spurious and are not associated with any true eigenvalues, even though a number of selected ones show a good convergence. This pathology, which occurs even in the case of smooth solutions, is already known (see \cite{boffi2010}) and it is the main motivation for the need of stabilization strategies implemented for nodal elements. 

\begin{figure}[!h]\setlength{\unitlength}{1cm}%\vspace{.5cm}
\begin{center}
\includegraphics[width=4.6cm, height=4.6cm]{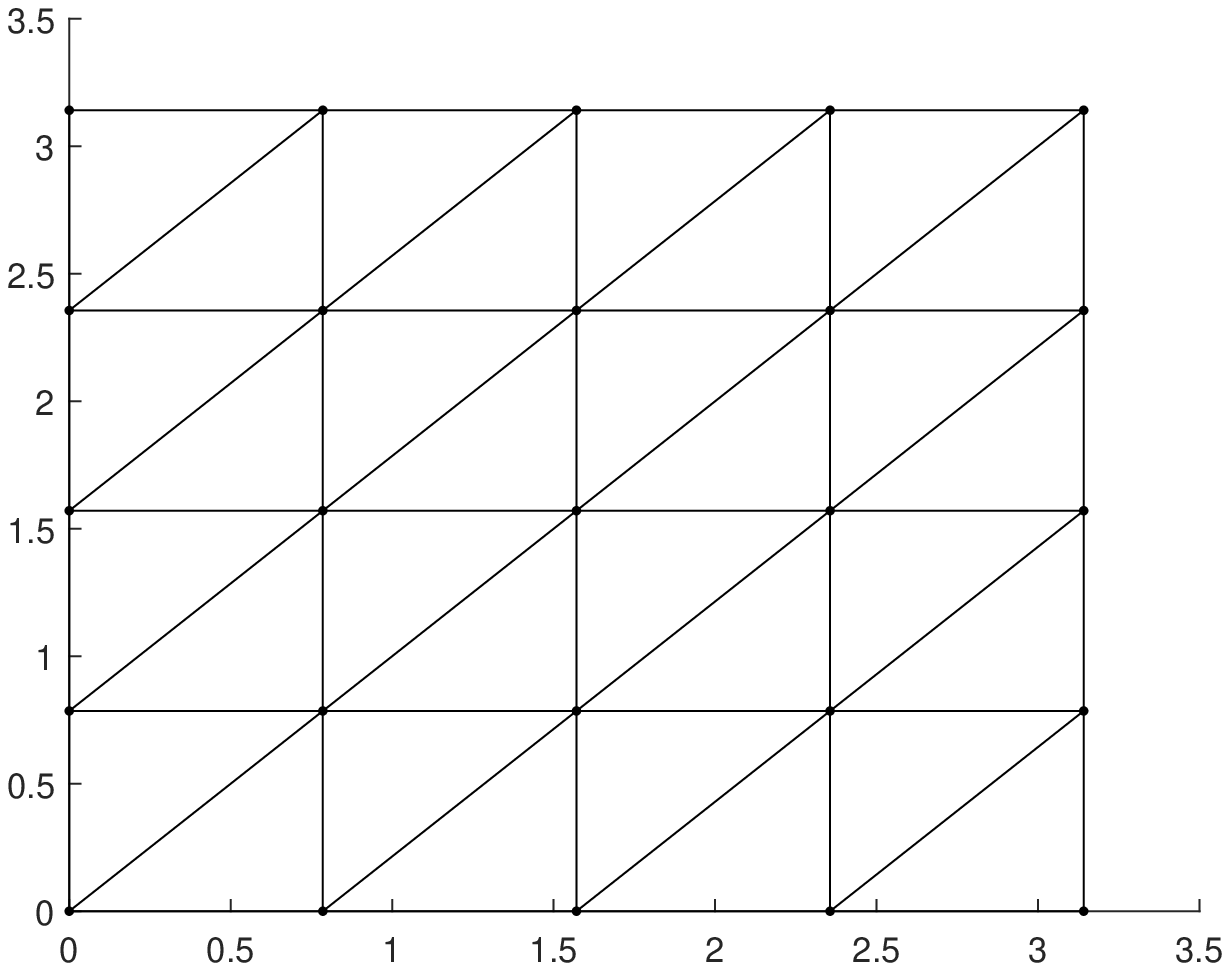}\hspace{1.5cm}
\includegraphics[width=4.6cm, height=4.6cm]{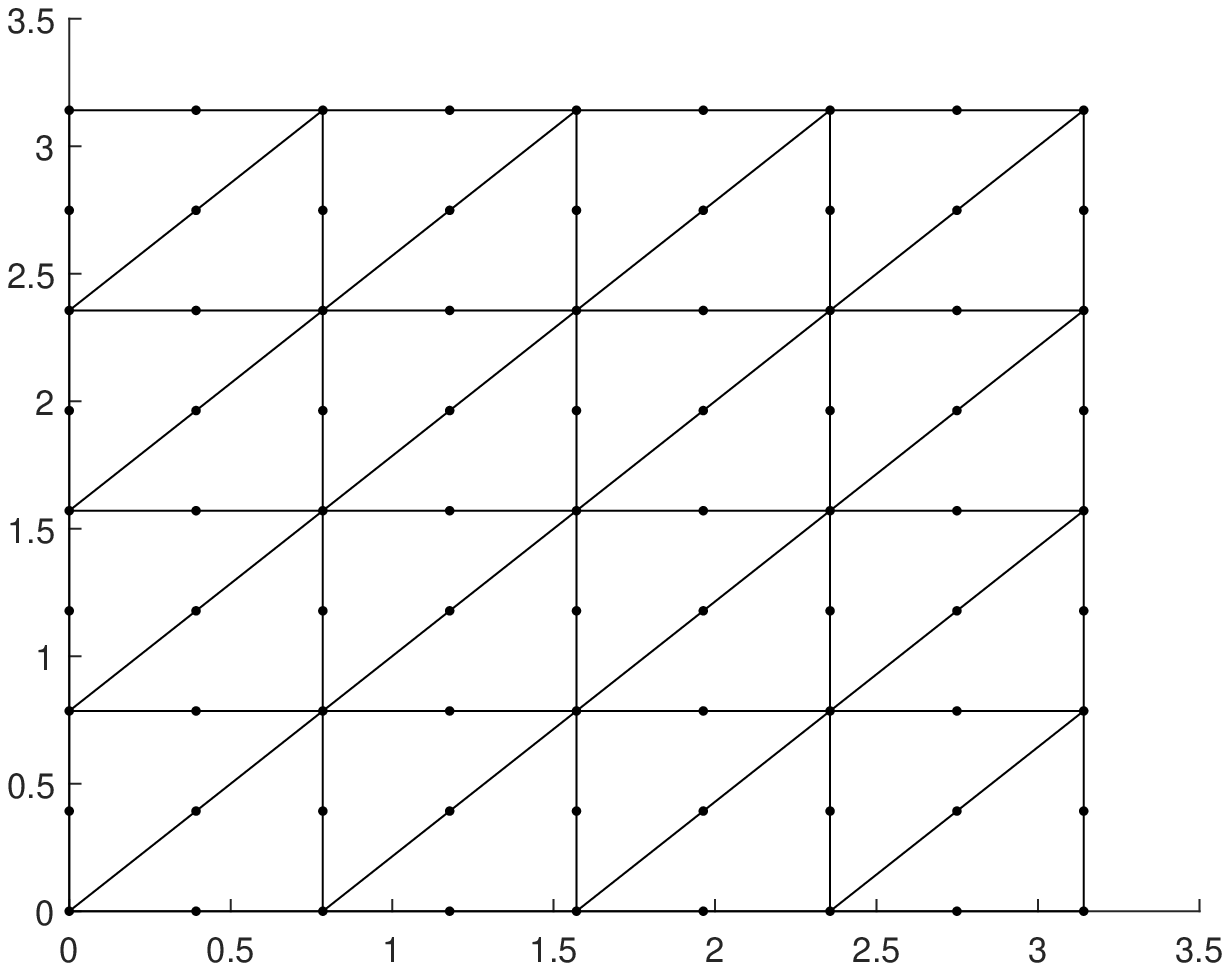}
\end{center}
\caption{Uniform meshes and node distributions used for $P_1$ (left) and  $P_2$ (right) interpolations on $\Omega_1$.}
\label{fig:sq_meshp1p2}
\end{figure}

\begin{table}[!h]
\centering
\small 
\caption{The first 17 eigenvalues on $\Omega_1$ using the SG formulation and $P_1$ elements on CC meshes.}  
\begin{tabular}{c|cccccc}
 \hline \hline  Exact  &    &       & Computed     &     &   
\tabularnewline		 & $N=5$   & $N=10$ & $N=15$  & $N=20$ & $N=25$
\tabularnewline  \hline  
  1.0000 &  1.0109 &  1.0027 (2.0) &  1.0012 (2.0) &  1.0007 (2.0)&  1.0004 (2.0) \\
 1.0000 &  1.0109 &  1.0027 (2.0) &  1.0012 (2.0) &  1.0007 (2.0)&  1.0004 (2.0) \\
 2.0000 &  2.0437 &  2.0110 (2.0) &  2.0049 (2.0) &  2.0027 (2.0)&  2.0018 (2.0) \\
 4.0000 &  4.1719 &  4.0437 (2.0) &  4.0195 (2.0) &  4.0110 (2.0)&  4.0070 (2.0) \\
 4.0000 &  4.1719 &  4.0437 (2.0) &  4.0195 (2.0) &  4.0110 (2.0)&  4.0070 (2.0) \\
 5.0000 &  5.2657 &  5.0683 (2.0) &  5.0304 (2.0) &  5.0171 (2.0)&  5.0110 (2.0) \\
 5.0000 &  5.2657 &  5.0683 (2.0) &  5.0304 (2.0) &  5.0171 (2.0)&  5.0110 (2.0) \\
 8.0000 &  5.7988 &  5.9507 (0.1) &  5.9781 (0.0) &  5.9877 (0.0)&  5.9921 (0.0) \\
 9.0000 &  8.6504 &  8.1746 (-1.2) &  8.0779 (-0.3) &  8.0438 (-0.1)&  8.0281 (-0.1) \\
 9.0000 &  9.8403 &  9.2197 (1.9) &  9.0982 (2.0) &  9.0554 (2.0)&  9.0355 (2.0) \\
10.0000 &  9.8403 &  9.2197 (-2.3) &  9.0982 (-0.4) &  9.0554 (-0.2)&  9.0355 (-0.1) \\
10.0000 & 10.9783 & 10.2710 (1.9) & 10.1213 (2.0) & 10.0684 (2.0)& 10.0438 (2.0) \\
13.0000 & 10.9783 & 10.2710 (-0.4) & 10.1213 (-0.1) & 10.0684 (-0.1)& 10.0438 (-0.0) \\
13.0000 & 12.5826 & 13.4573 (-0.1) & 13.2052 (2.0) & 13.1156 (2.0)& 13.0741 (2.0) \\
16.0000 & 12.5826 & 13.4573 (0.4) & 13.2052 (-0.2) & 13.1156 (-0.1)& 13.0741 (-0.1) \\
16.0000 & 14.3233 & 14.3101 (-0.0) & 14.6791 (0.6) & 14.8163 (0.4)& 14.8814 (0.3) \\
17.0000 & 14.3233 & 14.3101 (-0.0) & 14.6791 (0.4) & 14.8163 (0.2)& 14.8814 (0.1) \\
 \hline
\end{tabular}
\label{tab:o1_p1_standard_CC}
\end{table} 

\begin{table}[!h] 
\centering
%\begin{adjustbox}{width=1\textwidth}
\small 
\caption{The first 17 eigenvalues on $\Omega_1$ using the AG formulation and $P_1$ elements on CC meshes.}  
\begin{tabular}{c|cccccc}
 \hline \hline  Exact  &    &       & Computed     &     &   
\tabularnewline		 & $N=5$   & $N=10$ & $N=15$  & $N=20$ & $N=25$
\tabularnewline  \hline  %\rowcolor{lime}
 1.0000 &  1.0172 &  1.0032 (2.4) &  1.0013 (2.2) &  1.0007 (2.1)&  1.0005 (2.1) \\
 1.0000 &  1.0172 &  1.0032 (2.4) &  1.0013 (2.2) &  1.0007 (2.1)&  1.0005 (2.1) \\
 2.0000 &  2.0507 &  2.0115 (2.1) &  2.0050 (2.1) &  2.0028 (2.0)&  2.0018 (2.0) \\
 4.0000 &  4.2762 &  4.0516 (2.4) &  4.0211 (2.2) &  4.0115 (2.1)&  4.0072 (2.1) \\
 4.0000 &  4.2762 &  4.0516 (2.4) &  4.0211 (2.2) &  4.0115 (2.1)&  4.0072 (2.1) \\
 5.0000 &  5.3897 &  5.0756 (2.4) &  5.0318 (2.1) &  5.0176 (2.1)&  5.0111 (2.0) \\
 5.0000 &  5.3897 &  5.0756 (2.4) &  5.0318 (2.1) &  5.0176 (2.1)&  5.0111 (2.0) \\
 8.0000 &  8.8329 &  8.1838 (2.2) &  8.0796 (2.1) &  8.0444 (2.0)&  8.0283 (2.0) \\
 9.0000 & 10.3876 &  9.2606 (2.4) &  9.1066 (2.2) &  9.0581 (2.1)&  9.0366 (2.1) \\
 9.0000 & 10.3876 &  9.2606 (2.4) &  9.1066 (2.2) &  9.0581 (2.1)&  9.0366 (2.1) \\
10.0000 & 11.7852 & 10.3145 (2.5) & 10.1295 (2.2) & 10.0709 (2.1)& 10.0448 (2.1) \\
10.0000 & 11.7852 & 10.3145 (2.5) & 10.1295 (2.2) & 10.0709 (2.1)& 10.0448 (2.1) \\
13.0000 & 15.6090 & 13.5000 (2.4) & 13.2125 (2.1) & 13.1178 (2.0)& 13.0749 (2.0) \\
13.0000 & 15.6090 & 13.5000 (2.4) & 13.2125 (2.1) & 13.1178 (2.0)& 13.0749 (2.0) \\
16.0000 & 20.1558 & 16.8191 (2.3) & 16.3360 (2.2) & 16.1832 (2.1)& 16.1155 (2.1) \\
16.0000 & 20.1558 & 16.8191 (2.3) & 16.3360 (2.2) & 16.1832 (2.1)& 16.1155 (2.1) \\
17.0000 & 22.0376 & 17.9227 (2.4) & 17.3766 (2.2) & 17.2057 (2.1)& 17.1298 (2.1) \\  
 \hline
\end{tabular}
\label{tab:o1_p1_ag_CC}
\end{table} 
\begin{table}[!h] 
\centering
%\begin{adjustbox}{width=1\textwidth}
\small 
\caption{The first 17 eigenvalues on $\Omega_1$ using the OSGS formulation and $P_1$ elements on CC meshes.}  
\begin{tabular}{c|cccccc}
 \hline \hline  Exact  &    &       & Computed     &     &   
\tabularnewline		 & $N=5$   & $N=10$ & $N=15$  & $N=20$ & $N=25$
\tabularnewline  \hline  %\rowcolor{lime}
  1.0000 &  1.0165 &  1.0032 (2.4) &  1.0013 (2.2) &  1.0007 (2.1)&  1.0005 (2.1) \\
 1.0000 &  1.0165 &  1.0032 (2.4) &  1.0013 (2.2) &  1.0007 (2.1)&  1.0005 (2.1) \\
 2.0000 &  2.0492 &  2.0114 (2.1) &  2.0050 (2.0) &  2.0028 (2.0)&  2.0018 (2.0) \\
 4.0000 &  4.2653 &  4.0511 (2.4) &  4.0210 (2.2) &  4.0115 (2.1)&  4.0072 (2.1) \\
 4.0000 &  4.2675 &  4.0512 (2.4) &  4.0210 (2.2) &  4.0115 (2.1)&  4.0072 (2.1) \\
 5.0000 &  5.3726 &  5.0746 (2.3) &  5.0317 (2.1) &  5.0175 (2.1)&  5.0111 (2.0) \\
 5.0000 &  5.3726 &  5.0746 (2.3) &  5.0317 (2.1) &  5.0175 (2.1)&  5.0111 (2.0) \\
 8.0000 &  8.7944 &  8.1826 (2.1) &  8.0795 (2.1) &  8.0443 (2.0)&  8.0283 (2.0) \\
 9.0000 & 10.3443 &  9.2583 (2.4) &  9.1063 (2.2) &  9.0580 (2.1)&  9.0366 (2.1) \\
 9.0000 & 10.3443 &  9.2583 (2.4) &  9.1063 (2.2) &  9.0580 (2.1)&  9.0366 (2.1) \\
10.0000 & 11.7202 & 10.3087 (2.5) & 10.1287 (2.2) & 10.0707 (2.1)& 10.0448 (2.0) \\
10.0000 & 11.7499 & 10.3090 (2.5) & 10.1287 (2.2) & 10.0707 (2.1)& 10.0448 (2.0) \\
13.0000 & 15.5263 & 13.4941 (2.4) & 13.2117 (2.1) & 13.1177 (2.0)& 13.0749 (2.0) \\
13.0000 & 15.5263 & 13.4941 (2.4) & 13.2117 (2.1) & 13.1177 (2.0)& 13.0749 (2.0) \\
16.0000 & 19.9725 & 16.8118 (2.3) & 16.3350 (2.2) & 16.1829 (2.1)& 16.1154 (2.1) \\
16.0000 & 19.9917 & 16.8129 (2.3) & 16.3350 (2.2) & 16.1829 (2.1)& 16.1154 (2.1) \\
17.0000 & 21.1252 & 17.9058 (2.2) & 17.3740 (2.2) & 17.2050 (2.1)& 17.1296 (2.1) \\
 \hline
\end{tabular}
\label{tab:o1_p1_osgs_CC}
\end{table} 

Tables \ref{tab:o1_p1_ag_CC} and \ref{tab:o1_p1_osgs_CC} respectively list the analogous results obtained from the AG and OSGS formulations. These are obtained taking $\ell=0.1$, $c_{\u}=0.01$, and $c_p=0.6$. The significant role in the relief of the aforementioned pathology can easily be evidenced  from the correct values of the eigenvalues, and with the expected rates of convergence. 

Regarding the SG formulation, having arrived at the important conclusion that PS meshes should be used due to the potential risk of obtaining spurious eigenvalues otherwise, we present the results obtained from PS meshes when considering formulation \Eq{galerkin} in what follows. We tabulate the first 17 eigenvalues obtained using this formulation on PS meshes in Table \ref{tab:o1_p1_standard_PS} to numerically validate its convergence characteristics. 
 
\begin{table}[!h]
\centering
\small 
\caption{The first 17 eigenvalues on $\Omega_1$ using the SG formulation and $P_1$ elements on PS meshes.}  
\begin{tabular}{c|cccccc}
 \hline \hline  Exact  &    &       & Computed     &     &   
\tabularnewline		 & $N=5$   & $N=10$ & $N=15$  & $N=20$ & $N=25$
\tabularnewline  \hline  
1.0000 &  1.0029 &  1.0007 (2.0) &  1.0003 (2.0) &  1.0002 (2.0)&  1.0001 (2.0) \\
 1.0000 &  1.0072 &  1.0018 (2.0) &  1.0008 (2.0) &  1.0005 (2.0)&  1.0003 (2.0) \\
 2.0000 &  2.0197 &  2.0051 (2.0) &  2.0023 (2.0) &  2.0013 (2.0)&  2.0008 (2.0) \\
 4.0000 &  4.0792 &  4.0203 (2.0) &  4.0090 (2.0) &  4.0051 (2.0)&  4.0033 (2.0) \\
 4.0000 &  4.0796 &  4.0203 (2.0) &  4.0090 (2.0) &  4.0051 (2.0)&  4.0033 (2.0) \\
 5.0000 &  5.0772 &  5.0212 (1.9) &  5.0095 (2.0) &  5.0054 (2.0)&  5.0035 (2.0) \\
 5.0000 &  5.1596 &  5.0416 (1.9) &  5.0186 (2.0) &  5.0105 (2.0)&  5.0067 (2.0) \\
 8.0000 &  8.2651 &  8.0786 (1.8) &  8.0357 (1.9) &  8.0202 (2.0)&  8.0130 (2.0) \\
 9.0000 &  9.3628 &  9.0968 (1.9) &  9.0434 (2.0) &  9.0245 (2.0)&  9.0157 (2.0) \\
 9.0000 &  9.4040 &  9.1067 (1.9) &  9.0478 (2.0) &  9.0269 (2.0)&  9.0173 (2.0) \\
10.0000 & 10.4348 & 10.1242 (1.8) & 10.0560 (2.0) & 10.0317 (2.0)& 10.0203 (2.0) \\
10.0000 & 10.4494 & 10.1251 (1.8) & 10.0562 (2.0) & 10.0317 (2.0)& 10.0203 (2.0) \\
13.0000 & 13.4436 & 13.1522 (1.5) & 13.0699 (1.9) & 13.0397 (2.0)& 13.0255 (2.0) \\
13.0000 & 13.7494 & 13.2576 (1.5) & 13.1178 (1.9) & 13.0668 (2.0)& 13.0429 (2.0) \\
16.0000 & 17.0734 & 16.3173 (1.8) & 16.1433 (2.0) & 16.0810 (2.0)& 16.0520 (2.0) \\
16.0000 & 17.0912 & 16.3176 (1.8) & 16.1434 (2.0) & 16.0810 (2.0)& 16.0520 (2.0) \\
17.0000 & 17.9694 & 17.3329 (1.5) & 17.1518 (1.9) & 17.0860 (2.0)& 17.0552 (2.0) \\
 \hline
\end{tabular}
\label{tab:o1_p1_standard_PS}
\end{table} 

\begin{table}[!h]
\small
\caption{The first 10 eigenvalues on $\Omega_1$ using the OSGS formulation and $P_1$ elements on the uniform mesh shown in Figure~\ref{fig:sq_meshp1p2}.}  
\begin{center} 
\begin{tabular}{c|ccccc}
  \hline \hline   Exact  &    &    & Computed   &     &                  
		\tabularnewline  & $N=20$& $N=25$ & $N=30$ & $N=35$& $N=40$ 
\tabularnewline  \hline
 1.0000 &  1.0021 &  1.0013 (2.0) &  1.0009 (2.0) &  1.0007 (2.0)&  1.0005 (2.0) \\
 1.0000 &  1.0021 &  1.0013 (2.0) &  1.0009 (2.0) &  1.0007 (2.0)&  1.0005 (2.0) \\
 2.0000 &  2.0073 &  2.0044 (2.3) &  2.0030 (2.2) &  2.0021 (2.1)&  2.0016 (2.1) \\
 4.0000 &  4.0329 &  4.0210 (2.0) &  4.0146 (2.0) &  4.0107 (2.0)&  4.0082 (2.0) \\
 4.0000 &  4.0329 &  4.0211 (2.0) &  4.0146 (2.0) &  4.0107 (2.0)&  4.0082 (2.0) \\
 5.0000 &  5.0385 &  5.0239 (2.1) &  5.0164 (2.1) &  5.0119 (2.1)&  5.0091 (2.0) \\
 5.0000 &  5.0572 &  5.0351 (2.2) &  5.0239 (2.1) &  5.0173 (2.1)&  5.0131 (2.1) \\
 8.0000 &  8.1167 &  8.0706 (2.3) &  8.0475 (2.2) &  8.0342 (2.1)&  8.0258 (2.1) \\
 9.0000 &  9.1666 &  9.1065 (2.0) &  9.0739 (2.0) &  9.0543 (2.0)&  9.0416 (2.0) \\
 9.0000 &  9.1667 &  9.1066 (2.0) &  9.0740 (2.0) &  9.0543 (2.0)&  9.0416 (2.0) \\  
  \hline
\end{tabular}
\end{center}
\label{tab:o1_p1_osgs_uni}
\end{table}

\begin{table}[!h]
\small
\caption{The first 10 eigenvalues on $\Omega_1$ using the OSGS formulation and $P_2$ elements on the uniform mesh shown in Figure~\ref{fig:sq_meshp1p2}.} 
\begin{center} 
\begin{tabular}{c|ccccc}
 \hline \hline   Exact  &    &    & Computed   &     &                  
\tabularnewline  & $N=20$& $N=25$ & $N=30$ & $N=35$& $N=40$ 
\tabularnewline  \hline
 1.0000 &  1.0000 &  1.0000 (4.0) &  1.0000 (4.0) &  1.0000 (4.0)&  1.0000 (4.0) \\
 1.0000 &  1.0000 &  1.0000 (4.0) &  1.0000 (4.0) &  1.0000 (4.0)&  1.0000 (4.0) \\
 2.0000 &  2.0000 &  2.0000 (4.1) &  2.0000 (4.1) &  2.0000 (4.1)&  2.0000 (4.1) \\
 4.0000 &  4.0001 &  4.0000 (4.0) &  4.0000 (4.0) &  4.0000 (4.0)&  4.0000 (4.0) \\
 4.0000 &  4.0001 &  4.0000 (4.0) &  4.0000 (4.0) &  4.0000 (4.0)&  4.0000 (4.0) \\
 5.0000 &  5.0001 &  5.0000 (4.0) &  5.0000 (4.0) &  5.0000 (4.0)&  5.0000 (4.0) \\
 5.0000 &  5.0001 &  5.0000 (4.1) &  5.0000 (4.0) &  5.0000 (4.0)&  5.0000 (4.0) \\
 8.0000 &  8.0004 &  8.0002 (4.1) &  8.0001 (4.1) &  8.0000 (4.1)&  8.0000 (4.0) \\
 9.0000 &  9.0006 &  9.0002 (4.0) &  9.0001 (4.0) &  9.0001 (4.0)&  9.0000 (4.0) \\
 9.0000 &  9.0006 &  9.0002 (4.0) &  9.0001 (4.0) &  9.0001 (4.0)&  9.0000 (4.0) \\ 
  \hline
\end{tabular}
\end{center}
\label{tab:o1_p2_osgs_uni}
\end{table} 

Before closing this section, let us emphasize that both stabilized formulations produce exceptional results in the case of smooth solutions independent of the mesh types we have tested in this work. To further corroborate this, in Tables \ref{tab:o1_p1_osgs_uni} and \ref{tab:o1_p2_osgs_uni} we list the approximations to the first 10 eigenvalues using respectively $P_1$ and $P_2$ elements on standard uniform (right diagonal) meshes, a sample of which is shown in Figure \ref{fig:sq_meshp1p2}. These results are obtained from the OSGS formulation; however, they are very similar to the ones obtained from the AG formulation, which are not included for brevity. We can easily infer from these results that the stabilization achieves a double order of convergence without any spurious value, as it is anticipated from the theory. 

\subsection{The L-shaped domain}
\label{subsec:lsh}

In this subsection we want to examine a widely considered test case, e.g., in \cite{dauge2003,buffa2009,Bonito2011,du2020}, that is an L-shaped domain with a re-entrant corner, defined by $\Omega_2= \left] -1, 1\right[^2 \setminus \{[0,1]\times [-1,0]$\}. The CC and PS discretizations of this domain with $N=5$ are shown in Figure~\ref{fig:LmeshCCPS}. All of the reported values computed by a stabilized formulation concerning this domain are obtained using $\ell=0.3$, $c_{\u}=0.85$, and $c_p=0.5$. In the numerical results, we use the reference values taken from \cite{dauge2003} for comparison.

\begin{figure}[!h]\setlength{\unitlength}{1cm}%\vspace{.5cm}
\begin{center}
\includegraphics[width=4.6cm, height=4.6cm]{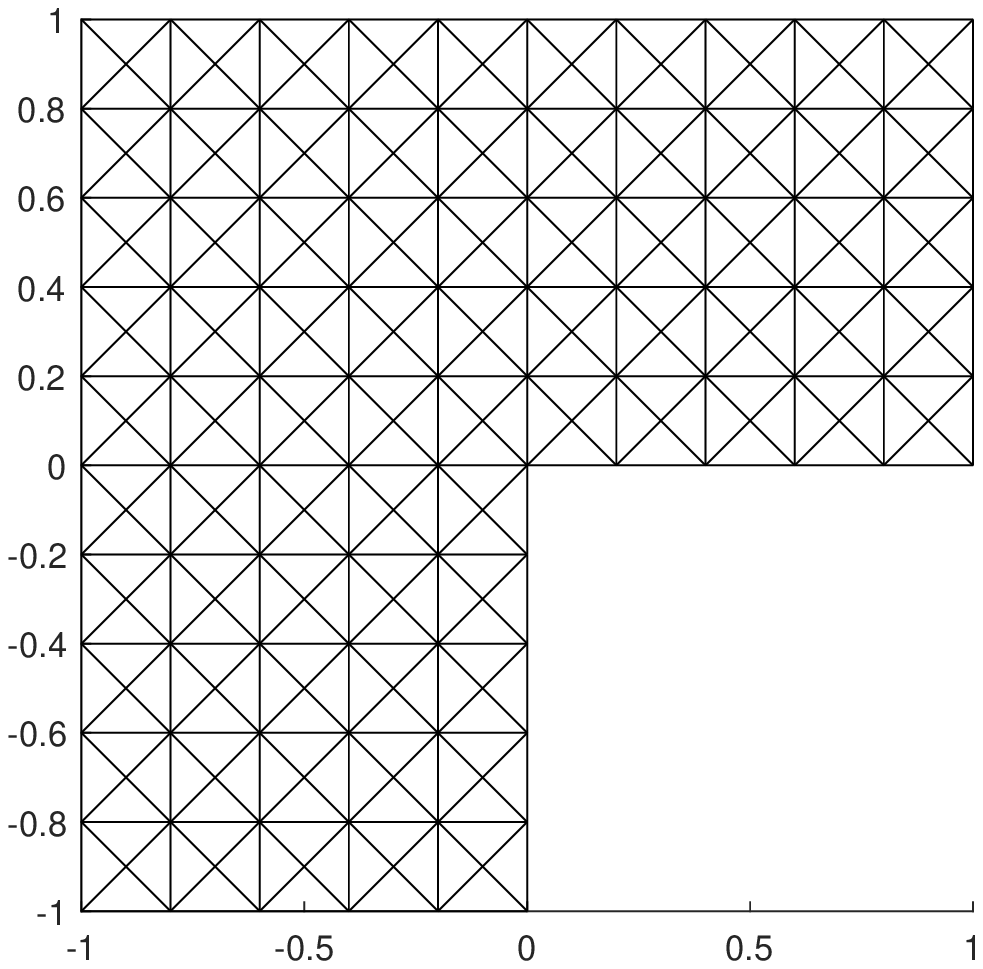}\hspace{1.5cm}
\includegraphics[width=4.6cm, height=4.6cm]{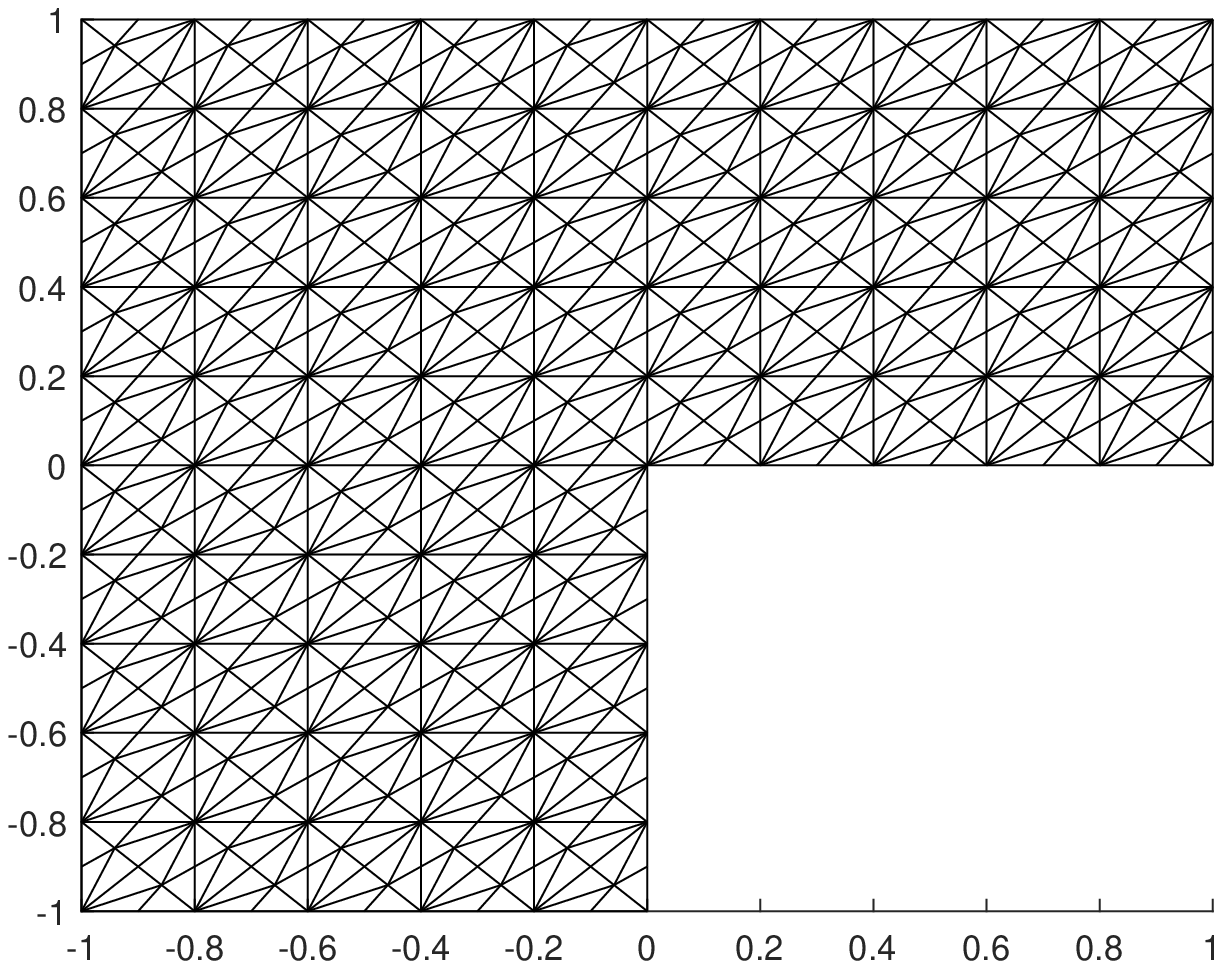}
\end{center}
\caption{A sample triangulation of  the L-shaped domain with CC (left) and PS (right) mesh, where $N=5$.}
\label{fig:LmeshCCPS}
\end{figure}

It is known that the first eigenvalue is the most critical when an approximation is considered, as it corresponds to the eigenfunction that has the lowest regularity, being it in $(H^{2/3-\epsilon}(\Omega))^2$ for any $\epsilon > 0$. When nodal elements are used, the existence of such a singularity manifests itself in the drastic change of the results depending on the way the normal vector is treated at the re-entrant corner in the process of boundary condition imposition. To realize and examine this computationally, we have tried three alternative ways of handling the components of the unknown vector field at the re-entrant corner of the enclosure; this is an issue in the case of nodal-based formulations. The first two alternatives are to force both of the components to vanish at the corner, or to leave them free at that node. The third one consists of assigning a fictitious normal vector due to the geometrical convenience. Specifically, this last strategy depends on assuming that the normal is the bisector of this corner, and imposing the boundary condition ${\bm n}  \times \u_h= \bm {0}$ in the following way. The tangent component of $\u$ is 0, and $\u$ has to follow the normal ${\bm n}$. In this way, if we write $\u=[u_1,u_2]^\mathsf{T}$, we impose the condition $u_2=-u_1$, and solve for $u_1$. 

We present the corresponding results obtained using PS meshes with the SG formulation in Table \ref{tab:rcorner_standPS} and the OSGS formulation in Table \ref{tab:rcorner_osgsPS}. These results clearly show that alternative strategies for enforcing the boundary condition at the re-entrant corner of the enclosure result in significantly different approximations in the first eigenvalue. The most remarkable deduction emanating from these tables is the influence of the treatment of the re-entrant corner with or without stabilization. It can be easily seen that leaving the components free at the corner lead to accurate results in the case of stabilization while it is not convenient at all for the SG formulation. The bisector normal strategy seems to work well for all of the formulations. Let us note here that despite of this difference, we have observed that the convergence properties are very similar. It is also of significant importance to note once again that the results obtained from the SG formulation seem to be more accurate than the ones obtained from the stabilized formulations. On the other hand, it is always possible to increase the accuracy in the latter schemes by manipulating the stabilization parameters, even though this has not been the main aim of this research. 

\begin{table}[!h]
\caption{The first 10 nonzero eigenvalues on $\Omega_2$ using PS mesh and the SG formulation, $N=9$.}  
\begin{center} 
\small
\begin{tabular}{ccc}
 \hline \hline   bisector normal  & $u_1=u_2=0$   &    $u_1,u_2$ free  
\tabularnewline  \hline 
   1.4876  &  1.4435   &     0.1181   \\
     3.5348 &   3.5348  &   1.4876\\
    9.8829 &  9.8829   &   3.5351\\
      9.8873 &     9.8873  &   9.8829\\
    11.4032 &    11.4032  &  10.3660 \\
    12.6424 &  12.4952   &   11.4038\\
     19.8006 &    19.8006  &  12.6424 \\
    21.5789 &    21.2304  &    20.3093\\
     23.4182  &      23.4182 &   21.5789\\
     28.7328 &   28.3622  &    23.4234\\
 \hline
\end{tabular}
\end{center}
\label{tab:rcorner_standPS}
\end{table} 

\begin{table}[!h]
\caption{The first 10 nonzero eigenvalues on $\Omega_2$ using PS mesh and the OSGS formulation, $N=9$.}  
\begin{center} 
\scriptsize
\begin{tabular}{ccc}
 \hline \hline   bisector normal  & $u_1=u_2=0$  & $u_1,u_2$ free
\tabularnewline  \hline 
  1.6252  &    2.2203   &   1.6252     \\
	3.5517&  3.5517  &   3.5371     \\
   9.8836 &   9.8836  &     9.7725    \\
	  9.8879	&  9.8879   &    9.8836      \\
  11.4116  & 11.4116   &     11.4058     \\
	12.6789	&   12.8182  &    12.6789     \\
    19.8034 &  19.8034  &    19.4922      \\
	  21.5959	&   21.6729  &    21.5959     \\
    23.4449 &  23.4449  &    23.4287     \\
	 28.5389	&    28.5389  &     28.5380    \\			
 \hline
\end{tabular}
\end{center}
\label{tab:rcorner_osgsPS} 
\end{table} 

\begin{table}[!h]
\centering
\small 
\caption{The first 5 eigenvalues on $\Omega_2$ using the SG formulation and $P_1$ elements on PS meshes.}  
\begin{tabular}{c|cccccc}
 \hline \hline  Ref.  &    &       & Computed     &     &   
\tabularnewline		 & $N=5$   & $N=10$ & $N=15$  & $N=20$ & $N=25$
\tabularnewline  \hline  
1.4756 &  1.5024 &  1.4860 (1.4) &  1.4816 (1.4) &  1.4797 (1.4)&  1.4786 (1.3) \\
 3.5340 &  3.5351 &  3.5347 (0.7) &  3.5344 (1.5) &  3.5342 (1.7)&  3.5342 (1.7) \\
 9.8696 &  9.9124 &  9.8804 (2.0) &  9.8744 (2.0) &  9.8723 (2.0)&  9.8713 (2.0) \\
 9.8696 &  9.9267 &  9.8839 (2.0) &  9.8760 (2.0) &  9.8732 (2.0)&  9.8719 (2.0) \\
11.3895 & 11.4314 & 11.4007 (1.9) & 11.3946 (1.9) & 11.3924 (2.0)& 11.3913 (2.0) \\
  \hline
\end{tabular}
\label{tab:o2_p1_st_PS}
\end{table} 

%the results when st formulation with u1=u2=0 (instead of bisector normal) is used: 
%1.4756 &  1.4063 &  1.4476 (1.3) &  1.4592 (1.3) &  1.4644 (1.3)&  1.4673 (1.3) \\
 %3.5340 &  3.5351 &  3.5347 (0.7) &  3.5344 (1.5) &  3.5342 (1.7)&  3.5342 (1.7) \\
 %9.8696 &  9.9124 &  9.8804 (2.0) &  9.8744 (2.0) &  9.8723 (2.0)&  9.8713 (2.0) \\
 %9.8696 &  9.9267 &  9.8839 (2.0) &  9.8760 (2.0) &  9.8732 (2.0)&  9.8719 (2.0) \\
%11.3895 & 11.4314 & 11.4007 (1.9) & 11.3946 (1.9) & 11.3924 (2.0)& 11.3913 (2.0) \\

\begin{table}[!h]
\centering
\small 
\caption{The first 5 eigenvalues on $\Omega_2$ using the AG formulation and $P_1$ elements on PS meshes.}  
\begin{tabular}{c|cccccc}
 \hline \hline  Ref.  &    &       & Computed     &     &   
\tabularnewline		 & $N=5$   & $N=10$ & $N=15$  & $N=20$ & $N=25$
\tabularnewline  \hline  
 1.4756 &  1.9220 &  1.6790 (1.1) &  1.5981 (1.2) &  1.5603 (1.3)&  1.5389 (1.3) \\
 3.5340 &  3.6020 &  3.5467 (2.4) &  3.5386 (2.5) &  3.5362 (2.5)&  3.5353 (2.5) \\
 9.8696 &  9.9197 &  9.8808 (2.2) &  9.8745 (2.1) &  9.8723 (2.0)&  9.8713 (2.0) \\
 9.8696 &  9.9335 &  9.8844 (2.1) &  9.8761 (2.0) &  9.8732 (2.0)&  9.8719 (2.0) \\
11.3895 & 11.4729 & 11.4066 (2.3) & 11.3965 (2.2) & 11.3933 (2.2)& 11.3918 (2.1) \\
 \hline
\end{tabular}
\label{tab:o2_p1_ag_PS}
\end{table} 

In the light of these investigations, the approximations for this domain case using PS meshes and following the bisector normal strategy are listed in Tables \ref{tab:o2_p1_st_PS}, \ref{tab:o2_p1_ag_PS}, and \ref{tab:o2_p1_osgs_PS} for the SG, AG, and OSGS formulations, respectively. From these tables we can see that the convergences rates are the ones expected from the theory.  Specifically, there are no spurious values encountered in any of the formulations considered. The smallest rate of convergence is observed for the first eigenvalue, whose corresponding eigenfunction has the lowest regularity. Moreover, the OSGS formulation seems to be more accurate in comparison with the AG formulation when implemented using the same set of stabilization parameters. 

\begin{table}[!h]
\centering
\small 
\caption{The first 5 eigenvalues on $\Omega_2$ using the OSGS formulation and $P_1$ elements on PS meshes.}  
\begin{tabular}{c|cccccc}
 \hline \hline  Ref.  &    &       & Computed     &     &   
\tabularnewline		 & $N=5$   & $N=10$ & $N=15$  & $N=20$ & $N=25$
\tabularnewline  \hline  
 1.4756 &  1.7762 &  1.6068 (1.2) &  1.5538 (1.3) &  1.5294 (1.3)&  1.5157 (1.3) \\
 3.5340 &  3.6057 &  3.5476 (2.4) &  3.5389 (2.5) &  3.5364 (2.5)&  3.5354 (2.5) \\
 9.8696 &  9.9191 &  9.8808 (2.1) &  9.8745 (2.1) &  9.8723 (2.0)&  9.8713 (2.0) \\
 9.8696 &  9.9326 &  9.8843 (2.1) &  9.8761 (2.0) &  9.8732 (2.0)&  9.8719 (2.0) \\
11.3895 & 11.4732 & 11.4070 (2.3) & 11.3967 (2.2) & 11.3933 (2.2)& 11.3919 (2.1) \\
 \hline
\end{tabular}
\label{tab:o2_p1_osgs_PS}
\end{table} 

\begin{table}[!h]
\centering
\small 
\caption{The first 5 eigenvalues on $\Omega_2$ using the OSGS formulation and $P_1$ elements on CC meshes.}  
\begin{tabular}{c|cccccc}
 \hline \hline  Ref.  &    &       & Computed     &     &   
\tabularnewline		 & $N=5$   & $N=10$ & $N=15$  & $N=20$ & $N=25$
\tabularnewline  \hline  
 1.4756 & 1.6350 & 1.5391 (1.3) & 1.5126 (1.3) & 1.5008 (1.3)  &1.4943 (1.3)  \\
3.5340 & 3.6248 & 3.5520 (2.3) & 3.5407 (2.4) & 3.5374 (2.4)  &3.5360 (2.4)   \\
9.8696 & 10.0023 & 9.8985 (2.2) & 9.8820 (2.1) & 9.8765 (2.0) & 9.8740 (2.0)  \\
9.8696 & 10.0032 & 9.8985 (2.2) & 9.8820 (2.1) & 9.8765 (2.0)  &9.8740 (2.0)   \\
11.3895 & 11.5771 & 11.4316 (2.2) & 11.4073 (2.1) & 11.3993 (2.1) & 11.3957 (2.1)   \\  
 \hline
\end{tabular}
\label{tab:L_cc_p1}
\end{table} 

The plots of the components of the fundamental eigenfunction, that is associated with the minimum eigenvalue, obtained using the OSGS formulation on the PS mesh, are presented in Figure \ref{fig:o2_minvector_os_p1_PS}. The approximations that are obtained from the bisector normal strategy show well the singularity near the re-entrant corner, corroborating the pattern of the eigenfunction that can be expected.

\begin{figure}[!h]\setlength{\unitlength}{1cm}%\vspace{.5cm}
\begin{center}
\includegraphics[width=7.2cm, height=7.2cm]{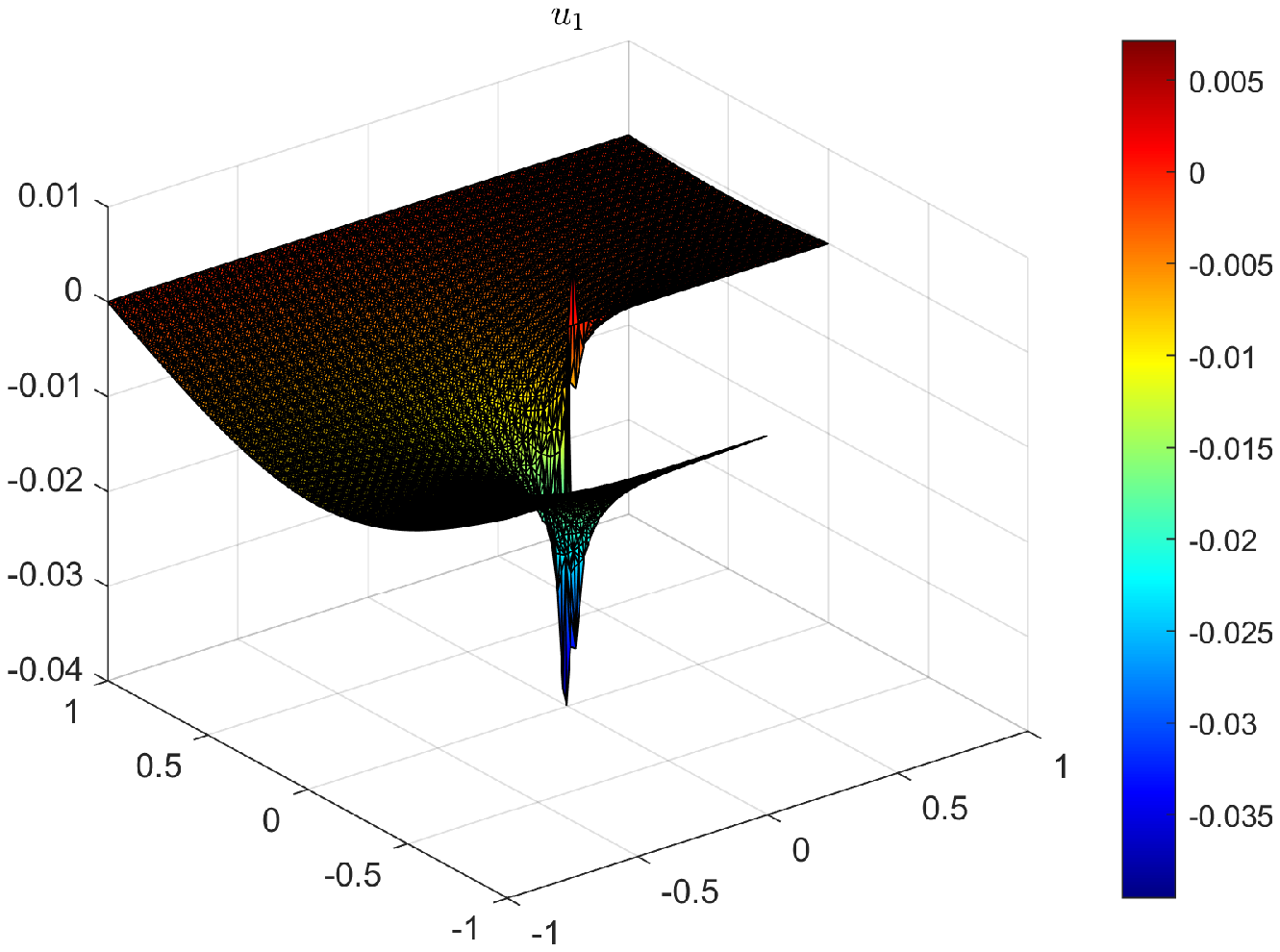}\hspace{1.0cm}
\includegraphics[width=7.2cm, height=7.2cm]{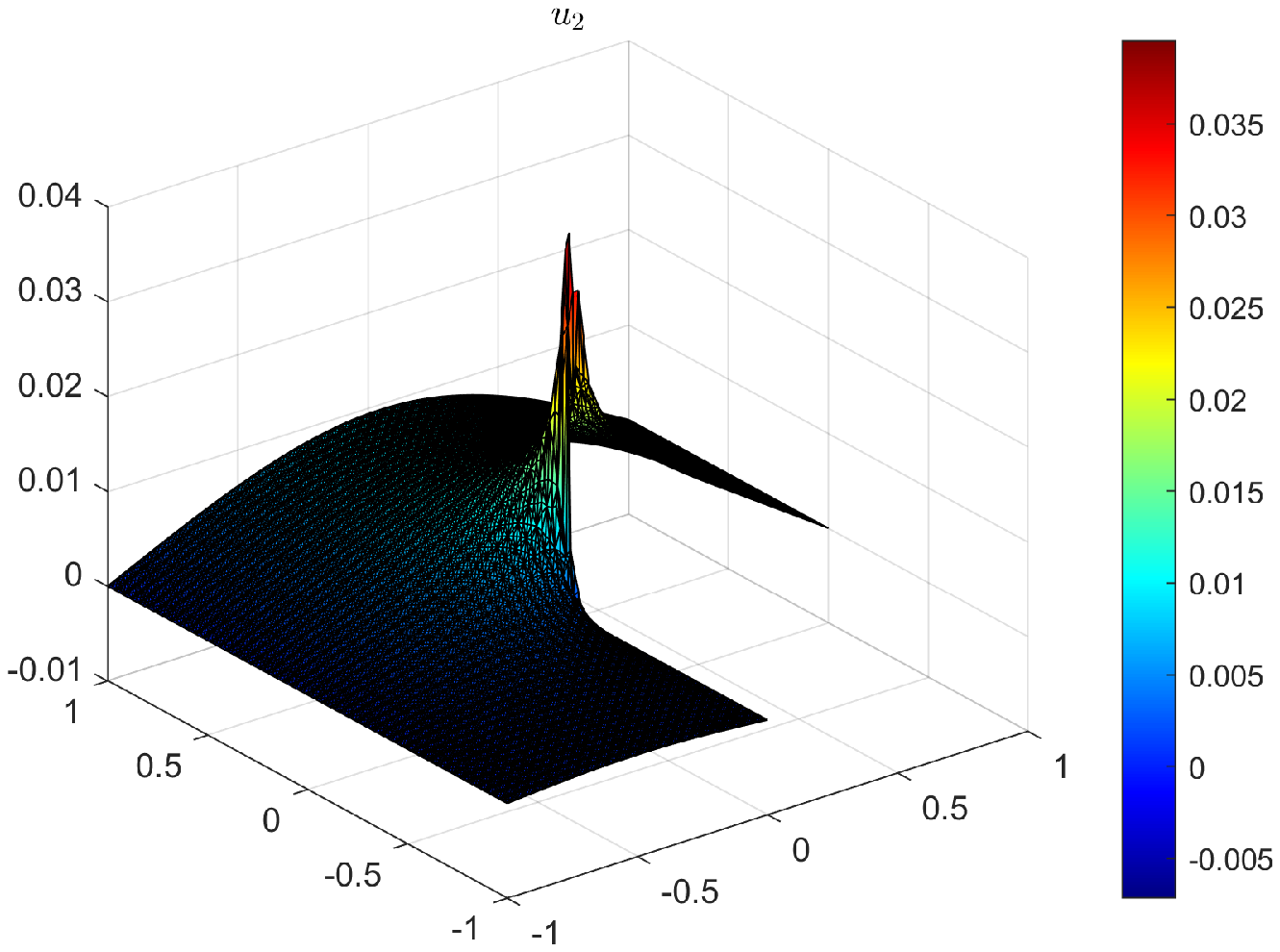}
\end{center}
\caption{The components of the fundamental eigenfunction on $\Omega_2$.}
\label{fig:o2_minvector_os_p1_PS}
\end{figure}

As a last illustration for this problem domain, we approximate the first 5 eigenvalues using the OSGS stabilized formulation on CC meshes. We list the results in Tables \ref{tab:L_cc_p1}  and  \ref{tab:L_cc_p2} for $P_1$ and $P_2$ elements, respectively. All the eigenvalues listed are approximated correctly, and the convergence rates are as expected. Note the unchanged rate in the first value with an increase in the order of interpolations due to the low regularity of the corresponding eigenfunction. Nevertheless, the accuracy is significantly improved for all the approximated values when quadratic interpolations are used instead of the linear ones.  The results mentioned here are in analogy with the ones obtained when the AG formulation is used in the simulations, and we prefer not to include them for conciseness of the presentation.    

\begin{table}[!h]
\centering
\small 
\caption{The first 5 eigenvalues on $\Omega_2$ using the OSGS formulation and $P_2$ elements on CC meshes.}  
\begin{tabular}{c|cccccc}
 \hline \hline  Ref.  &    &       & Computed     &     &   
\tabularnewline		 & $N=5$   & $N=10$ & $N=15$  & $N=20$ & $N=25$
\tabularnewline  \hline  
 1.4756 & 1.5446 & 1.5046 (1.3) & 1.4927 (1.3) & 1.4873 (1.3)& 1.4843 (1.3) \\
3.5340 & 3.5602 & 3.5388 (2.4) & 3.5357 (2.6) & 3.5348 (2.6)& 3.5345 (2.6) \\
9.8696 & 9.8701 & 9.8696 (4.1) & 9.8696 (4.0) & 9.8696 (4.0)& 9.8696 (4.0) \\
9.8696 & 9.8701 & 9.8696 (4.1) & 9.8696 (4.0) & 9.8696 (4.0)& 9.8696 (4.0) \\
11.3895 & 11.4010 & 11.3915 (2.5) & 11.3902 (2.6) & 11.3898 (2.6)& 11.3897 (2.7) \\ 
 \hline
\end{tabular}
\label{tab:L_cc_p2}
\end{table}

\subsection{The cracked square domain}
\label{subsec:crs}

As a final test case, we consider a square domain with a crack defined as $\Omega_3= \left] -1, 1\right[^2 \setminus \{(x,y) \in \mathbb{R} : 0 \leq x <1, y=0\}$. Sample discretizations of $\Omega_3$ are depicted in Figure \ref{fig:cs_meshCCPS}. In the sequel we report the results of our numerical simulations obtained by, unless it is otherwise stated, using linear interpolations on PS meshes, and taking $\ell=0.2$, $c_{\u}=0.1$, and $c_p=1.0$ for the stabilized formulations. We report also some results obtained by considering CC grids or quadratic interpolations on PS meshes in the sequel. As in the previous case, the reference eigenvalues are taken from \cite{dauge2003}. 

\begin{figure}[!h]\setlength{\unitlength}{1cm}%\vspace{.5cm}
\begin{center}
\includegraphics[width=4.6cm, height=4.6cm]{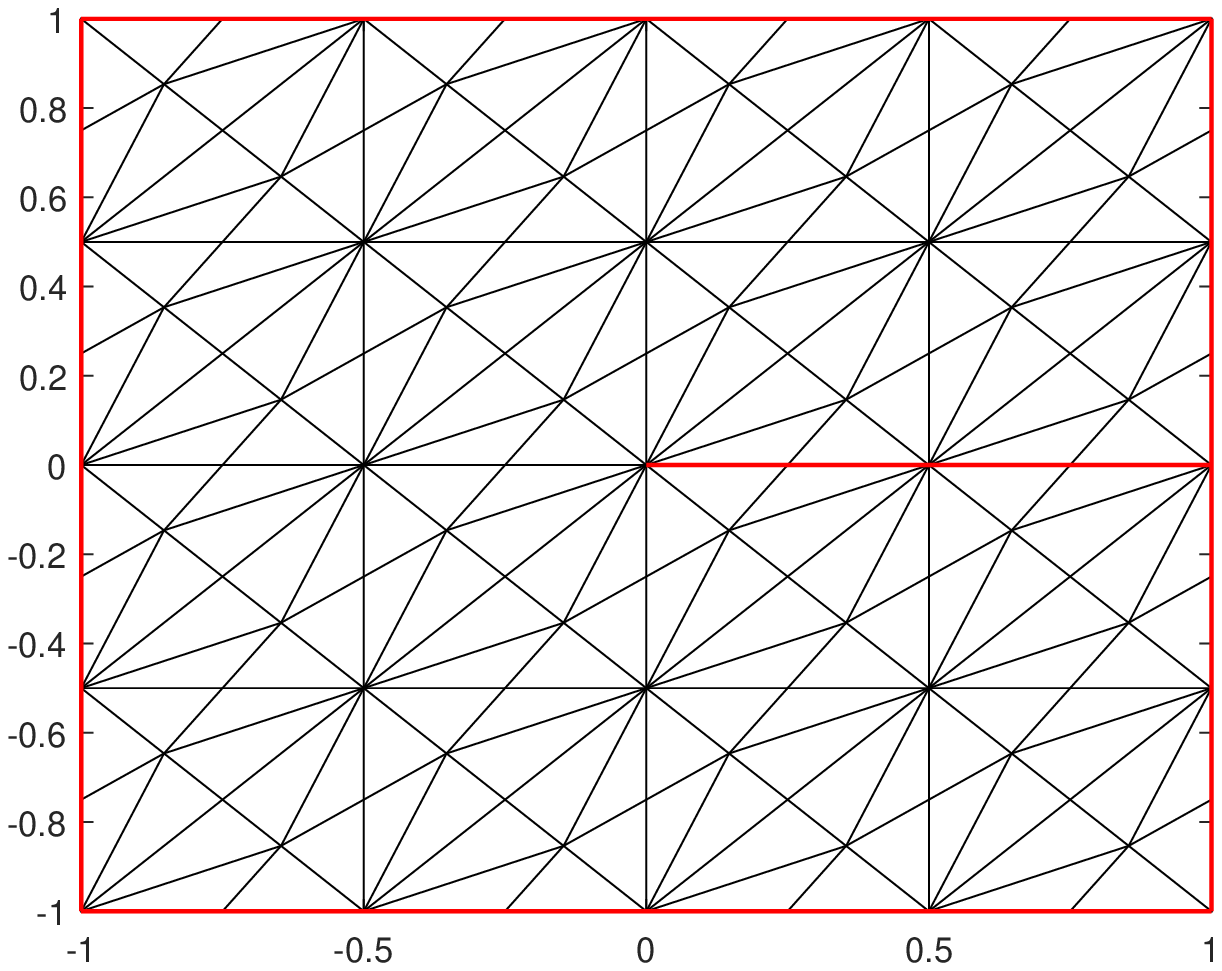}\hspace{1.5cm}
\includegraphics[width=4.6cm, height=4.6cm]{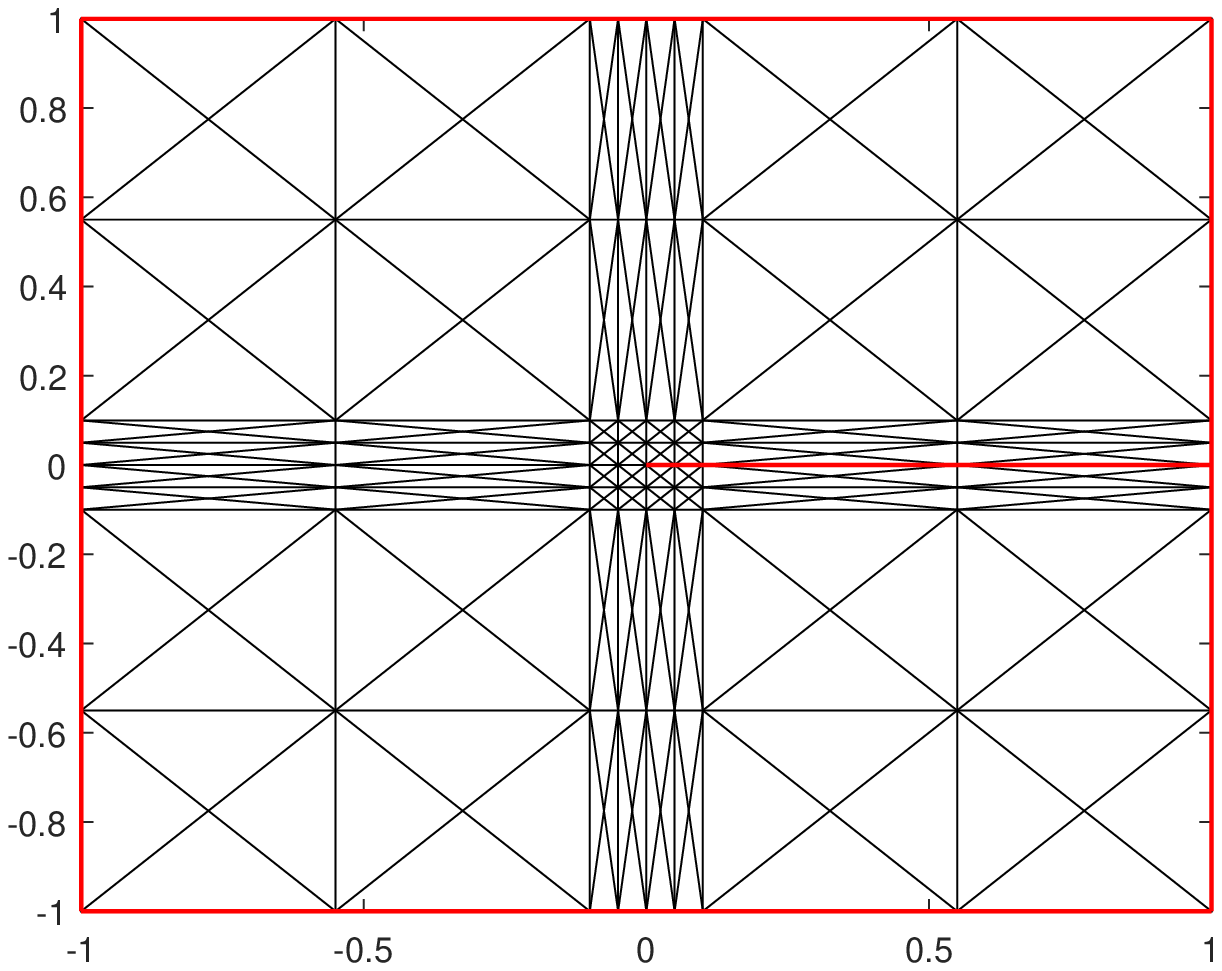}
\end{center}
\caption{Sample triangulations of the domain with a crack, $\Omega_3$, with PS mesh where $N=4$ (left) and internal layer CC mesh when $N=8$ (right). The boundaries are shown in red.}
\label{fig:cs_meshCCPS}
\end{figure}

Additionally to regular solutions, the EVP in this domain also has solutions that are unbounded near the tip of the slit, exhibiting a strong singularity. In particular, the smallest eigenvalue becomes the most crucial one to approximate, as its corresponding eigenfunction belongs to $(H^{1/2-\epsilon}(\Omega))^2$ for any $\epsilon > 0$ (see Figure \ref{fig:o3_minvector_os_p1_PS}). The same discussion about the treatment of the re-entrant corner as in the L-shape domain case applies to the tip of the crack for this problem, although without a clear identification of a fictitious normal in the present case. For the present instance, we examine the influence of the treatment of enforcing the boundary condition on the approximations by leaving the components of $\u$ as free or forcing them to vanish at the tip. We tabulate the corresponding results in Tables \ref{tab:o3_stand_PS_ufree} and \ref{tab:o3_osgs_PS_ufree} for the former, and Tables \ref{tab:o3_stand_PS_uzero} and \ref{tab:o3_osgs_PS_uzero} for the latter strategy. Tables \ref{tab:o3_stand_PS_ufree} and \ref{tab:o3_stand_PS_uzero} list the results of the SG formulation, whereas Tables \ref{tab:o3_osgs_PS_ufree} and \ref{tab:o3_osgs_PS_uzero} list the ones obtained from the OSGS stabilization. We can easily infer from these tables that the results are more accurate with higher convergence rates when the components of $\u_h$ are left to be free in comparison with the case where they are forced to vanish. The difference is very significant in the OSGS formulation, especially in the first eigenvalue, which is approximated with lowest accuracy. 

\begin{table}[!h] 
\small
\caption{The first 10 eigenvalues on $\Omega_3$, SG formulation, $u_1,u_2$ are free at the tip.}  
\begin{center} 
\begin{tabular}{c|cccccc}
 \hline \hline  Ref.  &    &       & Computed     &     &   
\tabularnewline		 & $N=2$   & $N=4$ & $N=8$  & $N=16$ & $N=32$
\tabularnewline  \hline  
1.0341 &  2.5316 &  2.4804 (0.1) &  2.4723 (0.0) &  2.2958 (0.5)&  2.0894 (0.8) \\
 2.4674 &  4.3066 &  3.3216 (1.1) &  2.6162 (4.3) &  2.4699 (14.2)&  2.4689 (2.3) \\
 4.0469 &  4.9140 &  4.0973 (4.1) &  4.0654 (2.5) &  4.0563 (2.4)&  4.0525 (2.3) \\
 9.8696 & 11.2261 & 10.1299 (2.4) &  9.9673 (2.4) &  9.9193 (2.3)&  9.8995 (2.3) \\
 9.8696 & 11.2416 & 10.1301 (2.4) &  9.9674 (2.4) &  9.9194 (2.3)&  9.8995 (2.3) \\
10.8449 & 12.1956 & 11.1009 (2.4) & 10.9444 (2.3) & 10.8966 (2.3)& 10.8763 (2.2) \\
12.2649 & 13.6625 & 12.5979 (2.1) & 12.4445 (1.5) & 12.3944 (1.1)& 12.3726 (0.8) \\
12.3370 & 15.0035 & 13.3779 (1.4) & 12.9190 (1.4) & 12.7404 (1.3)& 12.6421 (1.3) \\
19.7392 & 22.6499 & 20.3592 (2.2) & 20.0052 (2.1) & 19.8851 (2.1)& 19.8312 (2.1) \\
21.2441 & 26.0026 & 22.9060 (1.5) & 22.0858 (1.7) & 21.8061 (1.4)& 21.6679 (1.3) \\
\hline
\end{tabular}
\end{center}
\label{tab:o3_stand_PS_ufree}
\end{table} 

\begin{table}[!h] 
\small
\caption{The first 10 eigenvalues on $\Omega_3$, OSGS stabilization, $u_1,u_2$ free at the tip.}  
\begin{center} 
\begin{tabular}{c|cccccc}
 \hline \hline  Ref.  &    &       & Computed     &     &   
\tabularnewline		 & $N=2$   & $N=4$ & $N=8$  & $N=16$ & $N=32$
\tabularnewline  \hline   
 1.0341 &  2.6628 &  2.4998 (0.2) &  2.4729 (0.0) &  1.8854 (0.8)&  1.4921 (0.9) \\
 2.4674 &  3.9804 &  3.1997 (1.0) &  2.4807 (5.8) &  2.4687 (3.4)&  2.4677 (2.1) \\
 4.0469 &  6.0123 &  4.2908 (3.0) &  4.0858 (2.7) &  4.0536 (2.5)&  4.0482 (2.4) \\
 9.8696 & 15.2294 & 10.3782 (3.4) &  9.9642 (2.4) &  9.8904 (2.2)&  9.8746 (2.1) \\
 9.8696 & 16.1181 & 10.3849 (3.6) &  9.9647 (2.4) &  9.8904 (2.2)&  9.8746 (2.1) \\
10.8449 & 16.3441 & 11.4912 (3.1) & 10.9694 (2.4) & 10.8714 (2.2)& 10.8511 (2.1) \\
12.2649 & 17.9249 & 13.0336 (2.9) & 12.4783 (1.8) & 12.3689 (1.0)& 12.3447 (0.4) \\
12.3370 & 19.2569 & 13.8721 (2.2) & 12.8570 (1.6) & 12.5328 (1.4)& 12.3928 (1.8) \\
19.7392 & 24.5106 & 21.3360 (1.6) & 20.0952 (2.2) & 19.8212 (2.1)& 19.7591 (2.0) \\
21.2441 & 29.7499 & 22.0312 (3.4) & 21.7211 (0.7) & 21.4149 (1.5)& 21.3097 (1.4) \\
\hline
\end{tabular}
\end{center}
\label{tab:o3_osgs_PS_ufree}
\end{table} 
\begin{table}[!h]
\small
\caption{The first 10 eigenvalues on $\Omega_3$, SG formulation, $u_1=u_2=0$ at the tip.}  
\begin{center} 
\begin{tabular}{c|cccccc}
 \hline \hline  Ref.  &    &       & Computed     &     &   
\tabularnewline		 & $N=2$   & $N=4$ & $N=8$  & $N=16$ & $N=32$
\tabularnewline  \hline   
1.0341 &  0.5429 &  0.7464 (0.8) &  0.8771 (0.9) &  0.9519 (0.9)&  0.9920 (1.0) \\
 2.4674 &  2.0383 &  2.3690 (2.1) &  2.4438 (2.1) &  2.4616 (2.0)&  2.4660 (2.0) \\
 4.0469 &  4.0133 &  4.0807 (-0.0) &  4.0581 (1.6) &  4.0500 (1.9)&  4.0477 (1.9) \\
 9.8696 & 10.3409 & 10.1588 (0.7) &  9.9470 (1.9) &  9.8892 (2.0)&  9.8745 (2.0) \\
 9.8696 & 10.5393 & 10.1701 (1.2) &  9.9475 (1.9) &  9.8892 (2.0)&  9.8745 (2.0) \\
10.8449 & 10.7668 & 11.1508 (-2.0) & 10.9348 (1.8) & 10.8681 (2.0)& 10.8507 (2.0) \\
12.2649 & 11.0607 & 11.2899 (0.3) & 11.7001 (0.8) & 11.9595 (0.9)& 12.1060 (0.9) \\
12.3370 & 11.3086 & 12.4285 (3.5) & 12.3954 (0.6) & 12.3534 (1.8)& 12.3412 (2.0) \\
19.7392 & 14.0332 & 20.1339 (3.9) & 20.0343 (0.4) & 19.8167 (1.9)& 19.7588 (2.0) \\
21.2441 & 14.6242 & 20.5661 (3.3) & 20.4182 (-0.3) & 20.7411 (0.7)& 20.9697 (0.9) \\
\hline
\end{tabular}
\end{center}
\label{tab:o3_stand_PS_uzero}
\end{table} 
\begin{table}[!h]
\small
\caption{The first 10 eigenvalues on $\Omega_3$, OSGS stabilization, $u_1=u_2=0$ at the tip.}  
\begin{center} 
\begin{tabular}{c|cccccc}
 \hline \hline  Ref.  &    &       & Computed     &     &   
\tabularnewline		 & $N=2$   & $N=4$ & $N=8$  & $N=16$ & $N=32$
\tabularnewline  \hline   
 1.0341 &  5.7091 &  4.1253 (0.6) &  3.1774 (0.5) &  2.7209 (0.3)&  2.4921 (0.2) \\
 2.4674 &  7.3860 &  4.3581 (1.4) &  4.1048 (0.2) &  3.4704 (0.7)&  2.5474 (3.6) \\
 4.0469 & 10.2714 &  5.9888 (1.7) &  4.6276 (1.7) &  4.0582 (5.7)&  4.0488 (2.6) \\
 9.8696 & 17.5589 & 10.4587 (3.7) &  9.9724 (2.5) &  9.8910 (2.3)&  9.8746 (2.1) \\
 9.8696 & 18.2007 & 10.4646 (3.8) &  9.9729 (2.5) &  9.8910 (2.3)&  9.8746 (2.1) \\
10.8449 & 19.3851 & 11.6565 (3.4) & 10.9942 (2.4) & 10.8748 (2.3)& 10.8515 (2.2) \\
12.2649 & 23.7984 & 13.9401 (2.8) & 12.7777 (1.7) & 12.4642 (1.4)& 12.3714 (0.9) \\
12.3370 & 23.9740 & 15.4407 (1.9) & 13.7123 (1.2) & 13.0390 (1.0)& 12.6666 (1.1) \\
19.7392 & 34.1972 & 21.5708 (3.0) & 20.1193 (2.3) & 19.8234 (2.2)& 19.7592 (2.1) \\
21.2441 & 37.2805 & 24.2568 (2.4) & 22.5214 (1.2) & 21.8447 (1.1)& 21.5357 (1.0) \\
\hline
\end{tabular}
\end{center}
\label{tab:o3_osgs_PS_uzero}
\end{table} 

We have observed that all the resulting eigenfunctions related to the values we present are in physically meaningful agreement with the theoretical expectations, with an absence of any spurious mode in a checkerboard pattern. In Figure \ref{fig:o3_minvector_os_p1_PS}, we plot the components of the fundamental eigenfunction computed using the OSGS formulation when $N=32$. The components are left free at the tip of the crack (see Table \ref{tab:o3_osgs_PS_ufree}) in order to illustrate the singular behavior of the solution vector near it.  

\begin{figure}[!h]\setlength{\unitlength}{1cm}%\vspace{.5cm}
\begin{center}
\includegraphics[width=7.2cm, height=7.2cm]{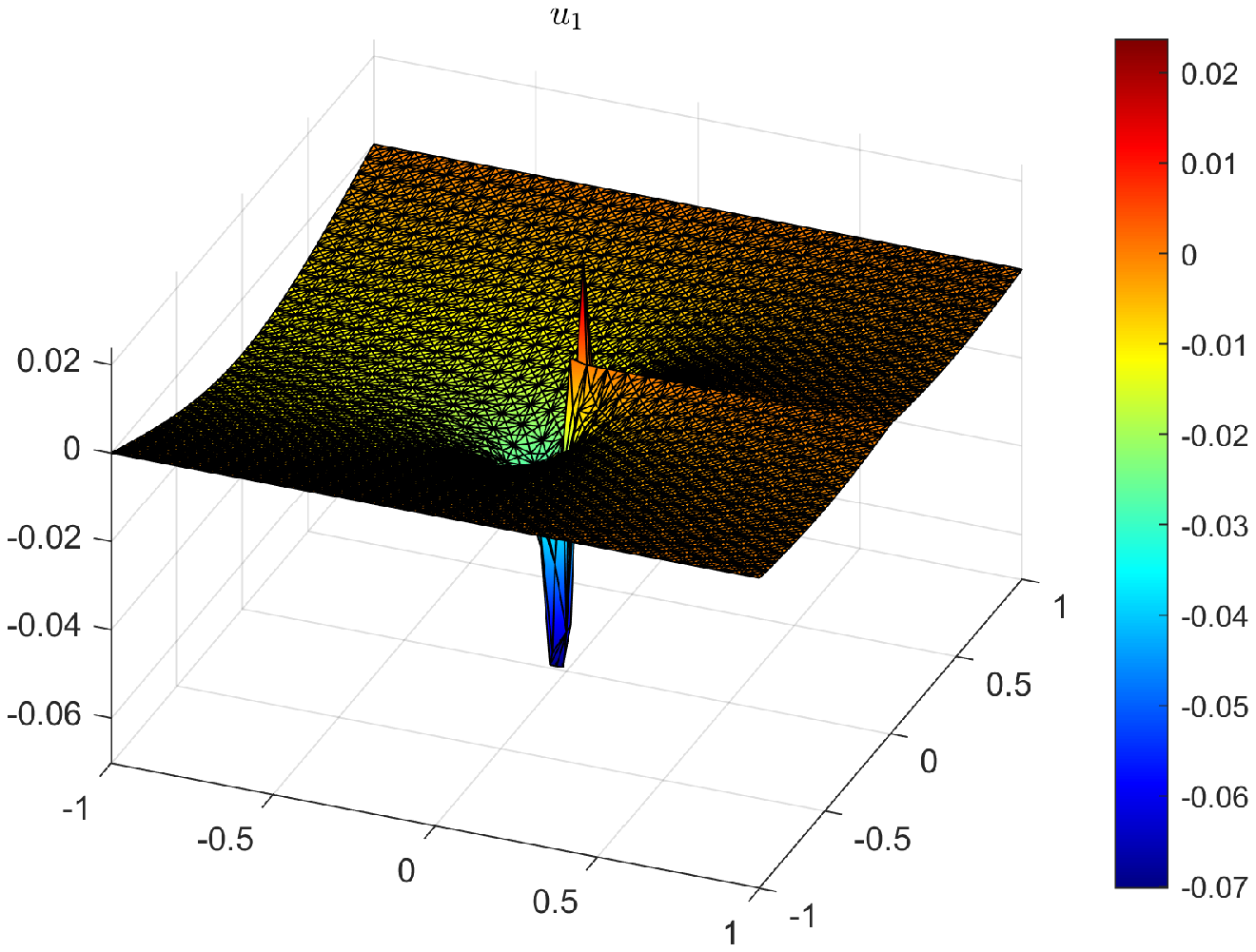}\hspace{1.0cm}
\includegraphics[width=7.2cm, height=7.2cm]{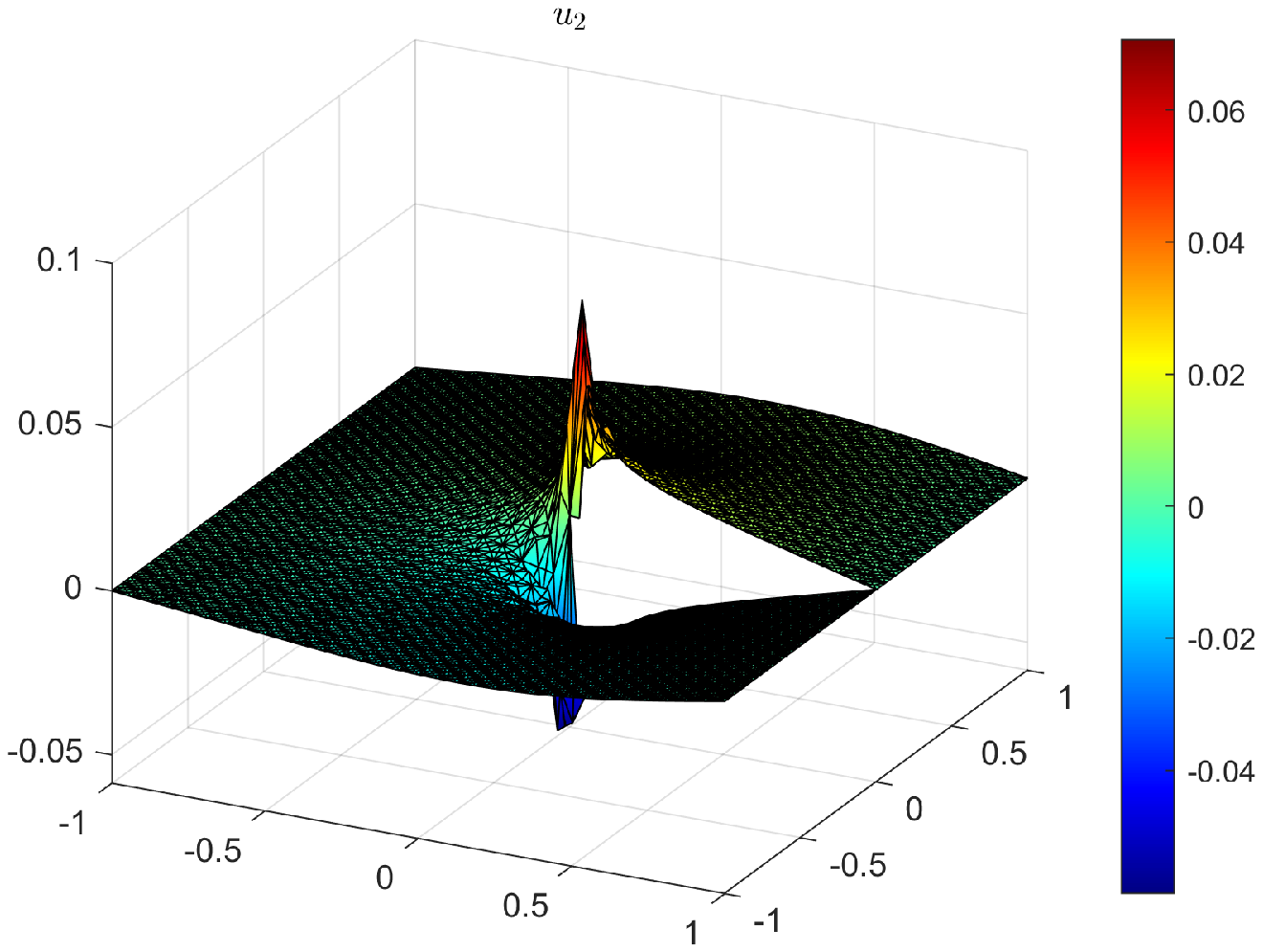}
\end{center}
\caption{The components of the fundamental eigenfunction on $\Omega_3$.}
\label{fig:o3_minvector_os_p1_PS}
\end{figure}

To compare the results obtained from different stabilizations for this domain, we list the first 10 eigenvalues obtained from the AG formulation when the components are left free and using the same set of stabilization parameters in Table \ref{tab:o3_ag_PS_ufree}. As before, the OSGS results remain more accurate for each eigenvalue, although a similar convergence tendency is observed in each one. 

\begin{table}[!h] 
\small
\caption{The first 10 eigenvalues on $\Omega_3$, AG stabilization, $u_1,u_2$ free at the tip.}  
\begin{center} 
\begin{tabular}{c|cccccc}
 \hline \hline  Ref.  &    &       & Computed     &     &   
\tabularnewline		 & $N=2$   & $N=4$ & $N=8$  & $N=16$ & $N=32$
\tabularnewline  \hline   
 1.0341 &  2.7081 &  2.5061 (0.2) &  2.4740 (0.0) &  2.4688 (0.0)&  1.9780 (0.6) \\
 2.4674 &  4.3613 &  4.0226 (0.3) &  3.6555 (0.4) &  2.7096 (2.3)&  2.4677 (9.6) \\
 4.0469 &  9.1407 &  5.1631 (2.2) &  4.1507 (3.4) &  4.0576 (3.3)&  4.0486 (2.7) \\
 9.8696 & 18.6120 & 10.4897 (3.8) &  9.9755 (2.5) &  9.8912 (2.3)&  9.8746 (2.1) \\
 9.8696 & 19.6901 & 10.4906 (4.0) &  9.9757 (2.5) &  9.8913 (2.3)&  9.8746 (2.1) \\
10.8449 & 20.0936 & 11.6341 (3.6) & 10.9921 (2.4) & 10.8741 (2.3)& 10.8513 (2.2) \\
12.2649 & 20.8839 & 13.2236 (3.2) & 12.5039 (2.0) & 12.3713 (1.2)& 12.3449 (0.4) \\
12.3370 & 29.5487 & 15.0561 (2.7) & 13.4124 (1.3) & 12.8012 (1.2)& 12.5249 (1.3) \\
19.7392 & 35.7446 & 21.7758 (3.0) & 20.1626 (2.3) & 19.8276 (2.3)& 19.7595 (2.1) \\
21.2441 & 41.9310 & 25.0012 (2.5) & 22.1829 (2.0) & 21.5450 (1.6)& 21.3602 (1.4) \\
\hline
\end{tabular}
\end{center}
\label{tab:o3_ag_PS_ufree}
\end{table} 

So far, we have reported the results from interpolations considered on PS type meshes. However, to show the capability of using CC type meshes with increased density along the crack and in the vicinity of the tip with the stabilized formulations, we have tested their approximation features on domains with such strong singularities. A sample triangulation is shown in Figure \ref{fig:cs_meshCCPS}. We list the results that are obtained from $P_1$ interpolations on these internal layer meshes taking $\ell=0.5$, $c_{\u}=2.0$, and $c_p=1.0$, using the AG formulation in Table \ref{tab:o3_AG_CC_ufree} and the OSGS formulation in Table \ref{tab:o3_osgs_CC_ufree}. The results in both of these tables put forward the overall acceptable convergence properties, noting as before the low regularity in the first eigenfunction.

\begin{table}[!h] 
\small
\caption{The first 10 eigenvalues on $\Omega_3$, AG stabilization using CC mesh, $u_1,u_2$ free at the tip.}  
\begin{center} 
\begin{tabular}{c|cccccc}
 \hline \hline  Ref.  &    &       & Computed     &     &   
\tabularnewline		 & $N=8$   & $N=16$ & $N=24$  & $N=32$ & $N=40$
\tabularnewline  \hline   
 1.0341 &  2.5302 &  2.4791 (0.1) &  2.4718 (0.0) &  2.4697 (0.0)&  2.2558 (0.7) \\
 2.4674 &  4.3128 &  3.9203 (0.3) &  3.0270 (2.4) &  2.5497 (6.7)&  2.4688 (18.3) \\
 4.0469 &  6.2468 &  4.1019 (5.3) &  4.0645 (2.8) &  4.0557 (2.4)&  4.0522 (2.3) \\
 9.8696 & 11.2056 & 10.1046 (2.5) &  9.9581 (2.4) &  9.9153 (2.3)&  9.8974 (2.2) \\
 9.8696 & 11.2080 & 10.1046 (2.5) &  9.9581 (2.4) &  9.9153 (2.3)&  9.8974 (2.2) \\
10.8449 & 12.2382 & 11.0886 (2.5) & 10.9381 (2.4) & 10.8936 (2.3)& 10.8748 (2.2) \\
12.2649 & 13.7696 & 12.5992 (2.2) & 12.4411 (1.6) & 12.3926 (1.1)& 12.3717 (0.8) \\
12.3370 & 16.4212 & 13.7661 (1.5) & 13.1375 (1.4) & 12.8756 (1.4)& 12.7339 (1.4) \\
19.7392 & 22.7601 & 20.3722 (2.3) & 20.0031 (2.2) & 19.8836 (2.1)& 19.8303 (2.1) \\
21.2441 & 26.3217 & 23.1452 (1.4) & 22.3451 (1.3) & 21.9772 (1.4)& 21.7905 (1.3) \\  
\hline
\end{tabular}
\end{center}
\label{tab:o3_AG_CC_ufree}
\end{table} 

\begin{table}[!h] 
\small
\caption{The first 10 eigenvalues on $\Omega_3$, OSGS stabilization using CC mesh, $u_1,u_2$ free at the tip.}  
\begin{center} 
\begin{tabular}{c|cccccc}
 \hline \hline  Ref.  &    &       & Computed     &     &   
\tabularnewline		 & $N=8$   & $N=16$ & $N=24$  & $N=32$ & $N=40$
\tabularnewline  \hline   
 1.0341 &  2.5316 &  2.4804 (0.1) &  2.4723 (0.0) &  2.2958 (0.5)&  2.0894 (0.8) \\
 2.4674 &  4.3066 &  3.3216 (1.1) &  2.6162 (4.3) &  2.4699 (14.2)&  2.4689 (2.3) \\
 4.0469 &  4.9140 &  4.0973 (4.1) &  4.0654 (2.5) &  4.0563 (2.4)&  4.0525 (2.3) \\
 9.8696 & 11.2261 & 10.1299 (2.4) &  9.9673 (2.4) &  9.9193 (2.3)&  9.8995 (2.3) \\
 9.8696 & 11.2416 & 10.1301 (2.4) &  9.9674 (2.4) &  9.9194 (2.3)&  9.8995 (2.3) \\
10.8449 & 12.1956 & 11.1009 (2.4) & 10.9444 (2.3) & 10.8966 (2.3)& 10.8763 (2.2) \\
12.2649 & 13.6625 & 12.5979 (2.1) & 12.4445 (1.5) & 12.3944 (1.1)& 12.3726 (0.8) \\
12.3370 & 15.0035 & 13.3779 (1.4) & 12.9190 (1.4) & 12.7404 (1.3)& 12.6421 (1.3) \\
19.7392 & 22.6499 & 20.3592 (2.2) & 20.0052 (2.1) & 19.8851 (2.1)& 19.8312 (2.1) \\
21.2441 & 26.0026 & 22.9060 (1.5) & 22.0858 (1.7) & 21.8061 (1.4)& 21.6679 (1.3) \\
\hline
\end{tabular}
\end{center}
\label{tab:o3_osgs_CC_ufree}
\end{table}

\section{Conclusions}
\label{sec:c}

We have studied and  numerically validated the characteristics of approximations to the solutions of the Maxwell eigenvalue problem that are obtained using nodal finite elements. Apart from the standard Galerkin formulation used with special (PS type) elements, two stabilized finite element formulations (AG and OSGS) have been implemented successfully to approximate both smooth and singular solutions. The convergence characteristics and error estimates rely on the associated analysis of the source problems. We have shown using the spectral theory that the formulations are optimally convergent for a set of algorithmic parameters that are implemented within the stabilized formulations. 

The Galerkin formulation, which is singular for the source problem, has been shown numerically to yield reasonable results using PS meshes, as expected. We have also shown that using CC type meshes may lead to spurious solutions even for smooth cases. The OSGS formulation has been shown to work proficiently on the meshes considered in this study, namely, PS and CC type meshes. The AG formulation has been shown to yield adequate results for smooth solutions, noting the sensitivity to strong singular solutions.  

As the main interest of the present study is in the use of nodal elements, a number of strategies of imposition of the boundary condition at the re-entrant corners have been explored. It has been set forth that while leaving the components free does not work well for the standard Galerkin formulation, it functions successfully  for both of the stabilized formulations. In addition, it has been shown that a fictitious normal may serve as the best alternative for some problem geometries such as the L-shape domain.  

Consequently, we have shown numerically that the stabilized methods can successfully approximate the eigensolutions of Maxwell's system when certain meshes are used, although with some limitations in accuracy in the case of strong singularities. The proposed methods compare very favorably with other formulations due to their ability to acquire the discrete spectrum without the obligation of eliminating the frequencies approximating zero, in addition to their facility of accommodating any order of interpolations and allowing a coupling of different operators.

\begin{acknowledgement}
R. Codina acknowledges the support received from the ICREA Acad\`emia Research Program of the Catalan Government.
\end{acknowledgement}

\newpage

\end{document}